\documentclass[12pt]{article}   
   
\usepackage{latexsym,amsfonts,amsmath,epsfig,tabularx,amsthm,amssymb,enumitem,bm}
\usepackage{float}

\usepackage{microtype}
   
\usepackage[draft=false,colorlinks,bookmarksnumbered,linkcolor=black, citecolor=black]{hyperref}

\usepackage{aliascnt}
\newtheorem{theorem}{Theorem}[section]

\newaliascnt{lem}{theorem}
\newtheorem{lemma}[lem]{Lemma}

\aliascntresetthe{lem}

\newaliascnt{prop}{theorem}
\newtheorem{prop}[prop]{Proposition}

\aliascntresetthe{prop}

\newaliascnt{cor}{theorem}

\aliascntresetthe{cor}

\newaliascnt{rem}{theorem}
\newtheorem{remark}[rem]{Remark}

\aliascntresetthe{rem}

\newaliascnt{def}{theorem}
\newtheorem{defi}[def]{Definition}

\aliascntresetthe{def}

\theoremstyle{definition}
\newaliascnt{ex}{theorem}

\aliascntresetthe{ex}     

\newaliascnt{ass}{theorem}
\newtheorem{ass}[ass]{Assumption}

\aliascntresetthe{ass}    
      
\renewcommand{\d}{\,\mathrm{d}}											
\renewcommand*{\epsilon}{\varepsilon}                                   
\renewcommand*{\rho}{\varrho}                                   		
\newcommand*{\nach}{\rightarrow}                                        
\newcommand*{\sep}{\; \vrule \;}                                        
\newcommand*{\N}{\mathbb{N}}                                            
\newcommand*{\R}{\mathbb{R}}                                            
\newcommand*{\C}{\mathbb{C}}                                            
\newcommand*{\Z}{\mathbb{Z}}                                            
\newcommand*{\B}{\mathcal{B}}                                           
\newcommand*{\U}{\mathcal{U}}                                         	
\newcommand*{\Co}{\mathcal{C}}                                         	
\newcommand*{\I}{\mathcal{I}}                                         	
\newcommand*{\J}{\mathcal{J}}                                         	
\renewcommand*{\P}{\mathcal{P}}                                         
\newcommand*{\leer}{\emptyset}                                          
\renewcommand*{\l}{\ell} 				                                

\newcommand*{\abs}[1]{\left| #1 \right|}                                
\newcommand*{\norm}[1]{\left\| #1 \right\|}                             
\newcommand*{\normmod}[1]{\left\|\hspace{-1pt}\left| #1 \right|\hspace{-1pt}\right\|}                             
\newcommand*{\link}[1]{(\ref{#1})}                                      
\newcommand*{\distr}[2]{\left\langle #1, #2 \right\rangle}              
\newcommand*{\dist}[2]{\mathrm{dist}\!\left( #1, #2 \right)}            

\renewcommand{\tilde}[1]{ \widetilde{#1} }        						
\DeclareMathOperator{\supp}{supp}										
\DeclareMathOperator{\rad}{rad}											
\DeclareMathOperator{\diam}{diam}										
\DeclareMathOperator{\Id}{Id}											
\DeclareMathOperator{\Ext}{Ext}											
\DeclareMathOperator{\Res}{Res}											

\usepackage{footmisc}

\usepackage{tikz}
\usetikzlibrary{arrows,decorations.pathmorphing,backgrounds,positioning,fit,petri}

\usepackage[margin=1.5cm]{caption}
   
\title{Besov regularity for operator equations on patchwise smooth manifolds \footnote{This work has been supported by Deutsche Forschungsgemeinschaft DFG (DA 360/19-1).}}     
 
\author{Stephan Dahlke \and Markus Weimar \footnote{Corresponding author.}} 

\begin{document}   
\maketitle

\begin{abstract}
\noindent We study regularity properties of solutions to operator equations on patchwise smooth manifolds~$\partial\Omega$ such as, e.g., boundaries of polyhedral domains $\Omega \subset \R^3$.
Using suitable biorthogonal wavelet bases $\Psi$, we introduce a new class of Besov-type spaces $B_{\Psi,q}^\alpha(L_p(\partial \Omega))$ of functions $u\colon\partial\Omega\nach\C$.
Special attention is paid on the rate of convergence for best $n$--term wavelet approximation to functions in these scales since this determines the performance of adaptive numerical schemes.
We show embeddings of (weighted) Sobolev spaces on $\partial\Omega$ into $B_{\Psi,\tau}^\alpha(L_\tau(\partial \Omega))$, $1/\tau=\alpha/2 + 1/2$, which lead us to regularity assertions for the equations under consideration.
Finally, we apply our results to a boundary integral equation of the second kind which arises from the double layer ansatz for Dirichlet problems for Laplace's equation in $\Omega$.

\smallskip
\noindent \textbf{Keywords:} \textit{Besov spaces, weighted Sobolev spaces, wavelets, adaptive methods, non-linear approximation, integral equations, double layer, regularity, manifolds.}

\smallskip
\noindent \textbf{AMS Subject Classification:} 30H25,  35B65, 42C40, 45E99, 46E35, 47B38, 65T60.
\end{abstract}

\section{Introduction}
In recent years, the adaptive numerical treatment of operator equations has become a field of increasing importance.  Motivated by the enormous increase of computer power, more and more complicated models have been developed, analyzed and simulated.  In practice, this might lead to systems involving hundreds
of thousands or even millions of unknowns.
Therefore, adaptive strategies are very often unavoidable to increase efficiency. In particular,   adaptive methods based on {\em wavelets}  seem to have a lot of potential~\cite{Da97}. Based on the strong analytical properties of wavelets,  in a series of papers \cite{CDD01}, \cite{CDD02}, \cite{DaDaHo+97}
adaptive algorithms have been derived  that are guaranteed to converge for a huge class of operators, including operators of negative order. Moreover, their  convergence order is optimal in the sense that  the algorithms realize  the convergence order of best $n$--term wavelet approximation schemes.   In the meantime, these investigations have been widely generalized, e.g., to wavelet frame algorithms \cite{St03}, \cite{StWe09}, to saddle point problems \cite{DaDaUr02}, and to non-linear equations \cite{CoDaDe03a}, \cite{K11}. 
This list is clearly not complete. 

Although numerical experiments  strongly indicate the potential of adaptive wavelet algorithms \cite{BBC+01}, \cite{W09}, a sound mathematical foundation for the use of adaptivity seems to be more than desirable. To ensure that adaptivity really pays off in practice, the convergence order of adaptive schemes  has to be compared to  the order of  non-adaptive (uniform) schemes. One the one hand, it is well-known that the power of uniform approximation schemes usually depends on the Sobolev smoothness of the object one wants to approximate. On the other hand, best $n$--term 
approximation schemes serve as the benchmark for adaptive strategies, and it is well-known that the approximation order that can be achieved by best $n$--term approximation algorithms  depends on the smoothness in  the specific scale (sometimes called the {\em adaptivity scale})   $B^s_{\tau}(L_{\tau}(\Omega)),~1/\tau=s/d+1/p$, of Besov spaces. We refer, e.g., to \cite{DDD97}, \cite{DNS06}, \cite{D98} for a detailed discussion of these relationships. Based on these facts,  we can now  make the following statement:  the use of adaptivity is justified if the smoothness of the unknown solution in the adaptivity scale of Besov spaces is higher than its Sobolev smoothness. In the realm of elliptic partial differential equations, a lot of positive results  have been established in the last years, see,  e.g., 
\cite{Da99b}, \cite{Da99c}, \cite{DD97}, \cite{DahDieHar+2014}, \cite{DaSi08}, \cite{DaSi13}, \cite{Han2014} 
(once again, this list is clearly not complete). In the meantime,  also first results for stochastic evolution equations  have been derived \cite{CDK+}.  However, to the best of our knowledge, no results in this direction are known for integral equations such as the double layer or the single layer potential operator.  
Therefore the aim of this paper is to fill this gap and to study the Besov smoothness of the solutions to integral equations on 
patchwise
smooth manifolds in order  to justify the use of adaptive wavelet algorithms. We think that this is an important issue since in many applications  such as the treatment of elliptic PDEs on unbounded, exterior domains the reduction to an integral equation on the boundary of the domain is very often the method of choice. We refer to the textbook \cite{SS11} for a detailed discussion. Moreover, it seems that integral equations are particularly suited for the treatment by adaptive wavelet schemes  since the vanishing moment property of wavelets can very efficiently be used to compress the usually densely populated system 
matrices; see, e.g., \cite{DHS07}, \cite{HS03}, \cite{Sch06}
for details. Indeed, in many cases, it has turned out that adaptive wavelet Galerkin boundary element methods (BEM)
outperform other powerful methods such as, e.g., multipole expansions \cite{OSW05}.

The starting point of our investigations has been the following observation: On smooth manifolds, usually no gain (except for constants) in the adaptivity scale of Besov spaces compared to the Sobolev scale can be expected since then the problem is completely regular.  Analogously to the case of elliptic PDEs on non-smooth domains, one would conjecture that the situation is quite different on a general Lipschitz manifold. However, in this case one is faced with a serious problem:  smoothness spaces such as Sobolev and Besov spaces constructed  by charts and partitions of unity are usually only well-defined for smoothness parameters  
$s \leq 1$; see \cite{T06}!  
The situation is slightly better for 
patchwise 
smooth manifolds 
such as, e.g., boundaries 
of polyhedral domains in $\R^3$; then the upper bound is given by $s<\min\{3/2, s_{\partial \Omega}\}$, where $s_{\partial \Omega}$ only depends on the Lipschitz character of the domain. We refer again to \cite{SS11}.
But even then, the usual definition of Besov spaces for  high smoothness parameters $s$  does not make any sense. One might define these spaces as abstract trace spaces \cite{CS10}, but then no intrinsic characterization is available and it is unclear how a wavelet characterization of these spaces should look like. We therefore proceed in a different way as we shall now explain.  Wavelets have the potential to characterize function spaces such as Besov spaces in the sense that weighted sequence norms of wavelet expansion coefficients are equivalent to smoothness norms \cite{FJ90}, \cite{HN09}, \cite{T06}. Exactly these norm equivalences have been  used to establish the already mentioned relationships of  best $n$--term approximation and Besov regularity  \cite{DNS06}, \cite{D98}, \cite{DJP92}. Therefore, given a suitable wavelet basis $\Psi$ on the manifold under consideration, we {\em define} new 
Besov-type spaces $B^\alpha_{\Psi, q}(L_p(\partial\Omega))$ 
as the sets of all functions for which the wavelet coefficients satisfy  specific  decay conditions. For small values of $s$, these spaces clearly coincide with classical Besov spaces. As we shall see in \autoref{sect:nterm},
membership in these spaces again implies a certain decay rate of best $n$--term approximation schemes with respect to the underlying wavelet basis. Consequently, if it would be possible to show that the smoothness of the solutions to integral equations in these generalized Besov scales is generically larger compared to  their Sobolev regularity, then the use of adaptive algorithms would be  completely justified.  The analysis presented in this paper shows that this is indeed the case, at least for patchwise smooth manifolds contained in $\R^3$; see \autoref{sect:general_eq} and \autoref{sect:doublelayer}.

The proof of this result uses  the following properties of solutions to integral equations. On patchwise smooth manifolds, singularities at the interfaces might occur which can 
diminish
the classical Sobolev smoothness of the solution.  Nevertheless, the behaviour of these   singularities can be controlled by means of specific {\em weighted} Sobolev spaces $X_{\rho}^k(\partial \Omega)$, where the weight is defined by some power of the distance to the interfaces. Under natural conditions, the solutions to integral equations are indeed contained in such  spaces; see, e.g., \cite{E92} and \autoref{prop:Elschner}.
Then, it turns out that the combination of  (low) classical Sobolev/Besov  smoothness and  (higher)  weighted Sobolev regularity implies  generalized Besov smoothness, i.e., an embedding of the form
\begin{equation*}
	B^s_{p}(L_p(\partial \Omega)) \cap X^k_{\rho}(\partial \Omega) \hookrightarrow B^{\alpha}_{\Psi, \tau}(L_{\tau}(\partial \Omega)),
	\qquad \frac{1}{\tau} = \frac{\alpha}{2}+\frac{1}{2},
\end{equation*}
holds. Since  $\alpha$ can be chosen significantly larger than $s$,  this result proves the claim.  The non-standard embeddings stated above are clearly of independent interest, and they constitute the main  result of this paper.

This paper is organized as follows: We start in 
\autoref{sect:boundaries}
with some preparations concerning the parametrizations of surfaces.  
In \autoref{sect:sobolev},
we introduce the weighted 
Sobolev spaces~$X^k_\rho(\partial\Omega)$, 
and we recall basic properties of wavelets on manifolds as far as it is needed for our purposes. 
In \autoref{sect:Besov}
we define our new 
Besov-type spaces $B^{\alpha}_{\Psi,q}(L_p(\partial\Omega))$
and clarify the relations to best $n$--term wavelet approximation. 
\autoref{sect:results}
contains the main results of this paper. 
In \autoref{sect:embeddings}
we state and prove the central non-standard embeddings 
mentioned above; see \autoref{thm:embedding}.
For the proof, the wavelet coefficients of a function in 
$B^s_{\Psi,p}(L_p(\partial \Omega)) \cap X^k_{\rho}(\partial \Omega)$
have to be estimated. This is performed by  combining  the vanishing moment property of wavelets and Whitney-type estimates with the weighted Sobolev regularity. 
In \autoref{thm:general_eq}, \autoref{sect:general_eq},
we apply the embedding results to general operator 
equations. Finally, \autoref{thm:doublelayer} in \autoref{sect:doublelayer} discusses a very 
important test example, i.e., the double layer potential for the Laplace operator.  
\autoref{sect:appendixA} contains the proofs of several technical lemmata and auxiliary propositions, 
whereas in \autoref{sect:appendixB} we show results which are essential ingredients for the proof of our main \autoref{thm:embedding}.

\textbf{Notation:} For families $\{a_{\J}\}_{\J}$ and $\{b_{\J}\}_{\J}$ of non-negative real numbers over a common index set we write $a_{\J} \lesssim b_{\J}$ if there exists a constant $c>0$ (independent of the context-dependent parameters $\J$) such that
\begin{equation*}
	a_{\J} \leq c\cdot b_{\J}
\end{equation*}
holds uniformly in $\J$.
Consequently, $a_{\J} \sim b_{\J}$ means $a_{\J} \lesssim b_{\J}$ and $b_{\J} \lesssim a_{\J}$.

\section{Parametrizations of surfaces}\label{sect:boundaries}
In this paper we restrict ourselves to Lipschitz surfaces $\partial\Omega$ which are boundaries of bounded, simply connected, closed polyhedra $\Omega\subset\R^3$ with finitely many flat, quadrilateral sides and straight edges; see \autoref{fig:surface}. 

In what follows we will pursue essentially two different (but equivalent) approaches to describe $\partial\Omega$ where both of them will be used later on. 
The first approach, based on a patchwise decomposition, is motivated by practical applications related to Computer Aided Geometric Design (CAGD).
On the other hand, we will also work with a description in terms of tangent cones which seems to be more appropriate for the analysis of functions on $\partial\Omega$.

Consider the decomposition 
\begin{equation}\label{partition}
	\partial\Omega = \bigcup_{i=1}^I \overline{F_i},
\end{equation}
where $\overline{F_i}$ denotes the closure of the $i$th (open) \emph{patch} of $\partial\Omega$ which is a subset of some affine hyperplane in $\R^3$, bounded by a closed polygonal chain connecting exactly four points (vertices of~$\Omega$).
Note that neither the domain $\Omega$ nor the quadrilaterals $F_i$ are assumed to be convex. In particular, some of the patches may possess a reentrant corner. However, we do require that the partition \link{partition} is essentially disjoint in the sense that the intersection of any two patches $\overline{F_i} \cap \overline{F_\ell}$, $i\neq \ell$, is either empty, a common edge, or a common vertex of $\Omega$.
Furthermore, we will assume the existence of (sufficiently smooth) diffeomorphic parametrizations
\begin{equation*}
	\kappa_i \colon [0,1]^2 \nach \overline{F_i}, \qquad i=1,\ldots,I,
\end{equation*}
which map the unit square onto these patches.
Finally, we define the class of \emph{patchwise smooth functions} on $\partial\Omega$ by
\begin{equation*}
	C_{\mathrm{pw}}^\infty(\partial\Omega) 
	:= \left\{ u \colon \partial\Omega \nach \C \sep u \text{ is globally continuous and } u\big|_{\overline{F_i}} \in C^{\infty}\!\left( \overline{F_i}\right) \text{ for all } i \right\}.
\end{equation*}

In the second approach the surface of $\Omega$ is modeled (locally) in terms of the boundary of its tangent cones $\Co_n$ subordinate to the vertices $\nu_1,\ldots,\nu_N$ of $\Omega$.
Recall that $x\in\R^3$ belongs to the (infinite) \emph{tangent cone} $\Co_n$ on $\Omega$, subordinate to the vertex $\nu_n$, if either $x=\nu_n$, or
\begin{equation*}
	\left\{x' = \nu_n + \gamma \, (x-\nu_n) \in \R^3 \sep \gamma \in [0,c) \right\} \subset \Omega
	\qquad \text{for some} \qquad 
	c>0.
\end{equation*}
The boundary of the cone $\Co_n$, $n\in\{1,\ldots,N\}$, consists of essentially disjoint, open plane sectors (called \emph{faces}) which will be denoted by $\Gamma^{n,1}, \ldots, \Gamma^{n,T_n}$, i.e.,
\begin{equation*}
	\partial\Co_n = \bigcup_{t=1}^{T_n} \overline{\Gamma^{n,t}}, \qquad n=1,\ldots,N.
\end{equation*}
We stress that for every $n$ the number of faces $T_n \geq 3$ can be arbitrary large (but finite). 
Moreover, note that for every pair $(n,t)$ with $n\in\{1,\ldots,N\}$ and $t\in\{1,\ldots,T_n\}$ there exists a uniquely defined patch number $i(n,t)\in\{1,\ldots,I\}$ such that
\begin{equation}\label{def_patch}
	\nu_n \in \overline{F_{i(n,t)}}
	\qquad \text{and} \qquad
	F_{i(n,t)} \cap \Gamma^{n,t} \neq \leer.
\end{equation}
It will be convenient to use local Cartesian (and/or polar) coordinates in each of the faces~$\Gamma^{n,t}$.
Formally that means we reparametrize every $\Gamma^{n,t}$ in $\R^3$ by some infinite plane sector 
\begin{equation*}
	\tilde{\Gamma^{n,t}} := \left\{ y = (r\,\cos(\phi), r\,\sin(\phi)) \in \R^2 \sep r\in(0,\infty), \phi\in(0,\gamma_{n,t}) \right\}
\end{equation*}
in $\R^2$ using a suitable, invertible mapping
\begin{equation*}
	R_{n,t} \colon \R^2 \supset \overline{\tilde{\Gamma^{n,t}}} \ni y \mapsto x\in  \overline{\Gamma^{n,t}} \subset \R^3.
\end{equation*}
Hence, $R_{n,t}$ is a composition of some extension, rotation and translation that maps $0\in\R^2$ onto the vertex $\nu_n$ of $\Omega$.
Observe that $R_{n,t}$ particularly preserves distances.

The advantage of the latter change of coordinates is that now functions $f_n \colon \partial\Co_n \nach \C$ can be described by a collection of functions $(f_{n,1},\ldots,f_{n,T_n})$, where
\begin{equation}\label{facewise_fkt}
	f_{n,t} := f_n\big|_{\overline{\Gamma^{n,t}}} \circ R_{n,t} \colon \overline{\tilde{\Gamma^{n,t}}} \nach \C, \qquad t=1,\ldots,T_n.
\end{equation}
The following class of (complex-valued) \emph{facewise smooth functions} with compact support in $\partial\Co_n$ will be of some importance in the sequel:
\begin{equation*}
	C_{0,\mathrm{fw}}^\infty(\partial\Co_n) 
	:= \left\{ f_n \in C_0(\partial\Co_n) \sep f_{n,t} \in C^{\infty}\!\left( \overline{\tilde{\Gamma^{n,t}}}\right) \text{ for all } t=1,\ldots,T_n \right\}.
\end{equation*}
In order to analyze functions defined on the whole surface $\partial\Omega$ we localize them to cone faces near the vertices of $\Omega$, using a special resolution of unity.

A \emph{special resolution of unity} on $\partial\Omega$ subordinate to the vertices $\nu_1,\ldots,\nu_N$ is a collection of non-negative functions $(\varphi_n)_{n=1}^N \subset C^\infty_{\mathrm{pw}}(\partial\Omega)$ with the following features.
Apart from the usual resolution property,
\begin{equation}\label{resolution}
	\sum_{n=1}^N \varphi_n(x)=1\quad \text{for all} \quad x\in\partial\Omega,
\end{equation}
we moreover assume that for all $n\in\{1,\ldots,N\}$ there exist open neighborhoods
\begin{equation*}
	\nu_n\in\U_{n,0} \subsetneqq \U_{n,1} \subsetneqq \partial\Omega
\end{equation*}
such that
\begin{enumerate}[label=(U\arabic{*}), ref=U\arabic{*}]
	\item \label{U1} $\dist{\U_{n,1}}{F_\ell} \geq C_1$ \quad for all \quad $\ell\in\{1,\ldots,I\}$ \quad with \quad $\nu_n \notin \overline{F_\ell}$,
	\item \label{U2} $\min\!\left\{ \dist{\U_{n,1}\cap \Gamma^{n,t}}{F_{i(n,t)} \setminus \Gamma^{n,t}} \sep t=1,\ldots,T_n \right\} \geq C_1$,
\end{enumerate}
and
\begin{equation}\label{eq:phi}
	\varphi_n(x) 
	= \begin{cases}
		1, &x\in \U_{n,0},\\
		0, &x\in \partial\Omega \setminus \U_{n,1}.
	\end{cases}	
\end{equation}
Therein $C_1$ is some small, positive constant and $\dist{M_1}{M_2}$ denotes the shortest distance between the sets $M_1$ and $M_2$ along their common superset (which is $\partial\Omega$ or some affine hyperplane, depending on the context). We visualized \link{U1} and \link{U2} by small arrows in the second picture of \autoref{fig:surface} below.  

Note that \link{U1} immediately implies the essentially disjoint representation
\begin{equation}\label{U3}
	\U_{n,1} =\bigcup_{t=1}^{T_n} \left( \U_{n,1} \cap \overline{\Gamma^{n,t}} \right)
\end{equation}
for all $n\in\{1,\ldots,N\}$. Moreover, combining \link{resolution} and \link{eq:phi} yields
\begin{equation*}
	\partial\Omega = \bigcup_{n=1}^N \U_{n,1}.
\end{equation*}
\begin{figure}[ht]
	\begin{center}
\begin{tikzpicture}
\fill [color=gray!20] (0,0)--(6,-0.5)--(-1,1.5)--(0.5,0.5)--(0,0);
\fill (0,0) circle (2pt) ;
\fill (6,-0.5)  circle (2pt);
\fill [white] (0.5,0.5) circle (2pt);
\draw (0.5,0.5) circle (2pt);
\fill (-1,1.5) circle (2pt);
\draw (0,0)--(6,-0.5);
\draw [dashed] (6,-0.5)--(-1,1.5);
\draw [dashed] (-1,1.5)--(0.45,0.55);
\draw [dashed] (0.45,0.45)--(0,0);
\draw (-1,1.5)--(0,0.84);

\fill (6,4)  circle (2pt);
\draw (6,-0.5)--(6,4);
\fill (-1,3.5)  circle (2pt);
\draw (-1,1.5)--(-1,3.5);
\fill (0,2.5)  circle (2pt);
\draw (0,0)--(0,2.5);
\fill (0.5,3.2) circle (2pt);
\draw [dashed] (0.5,0.57)--(0.5,3.2);

\draw (6,4)--(0,2.5)--(0.5,3.2)--(-1,3.5)--(6,4);
\end{tikzpicture}
		\hfill
\begin{tikzpicture}
\fill [color=gray!20] (2,1.5)--(2,3.5)--(3.05,2.25)--(2,1.5);
\fill (0,0) circle (2pt) node[anchor=north] {$\nu_n$};
\fill (5,0) circle (2pt);
\fill (2,1.5) circle (2pt);
\fill (2,3.5) circle (2pt);
\draw (0,0)--(5,0)--(2,3.5)--(2,1.5)--(0,0);

\draw [thick, dotted] (0,0) circle [radius=1.1];
\draw [thick] (1.1,0) arc (0:38:1.1); 
\draw (0,0) node[anchor=south east] {$\U_{n,0}$};

\draw [thick, dotted] plot [smooth] coordinates {(1.25,1) (1.2, 1.2) (1, 1.5) (0.4, 2.1)};
\draw [thick] plot [smooth] coordinates {(1.25,1) (1.35,0.8) (2,0.75) (2.75,1.4) (3.1,1.4) (4,0.5) (4,0)};
\draw [thick, dotted] plot [smooth] coordinates {(4,0) (3.8,-0.5) (3.6,-0.8)};
\draw (3,0) node[anchor=south] {$\U_{n,1}$};

\draw [dashed] (2,1.5) -- (5,3.75);
\draw [dashed] (5,0) -- (5.6,0);
\draw (4.7,2) node[anchor=west] {$\Gamma^{n,t}$};

\draw [<->,>=stealth'] (3.81,0.77) -- (4.15,0.9);
\draw [<->,>=stealth'] (2.9,1.55) -- (2.8,2);
\end{tikzpicture}
	\end{center}
	\caption{Boundary of some polyhedron $\Omega\subset\R^3$ with reentrant corners (left) and non-convex patch $F_{i(n,t)}$ illustrating conditions \link{U1} and \link{U2} of a special resolution of unity (right).}\label{fig:surface}
\end{figure}
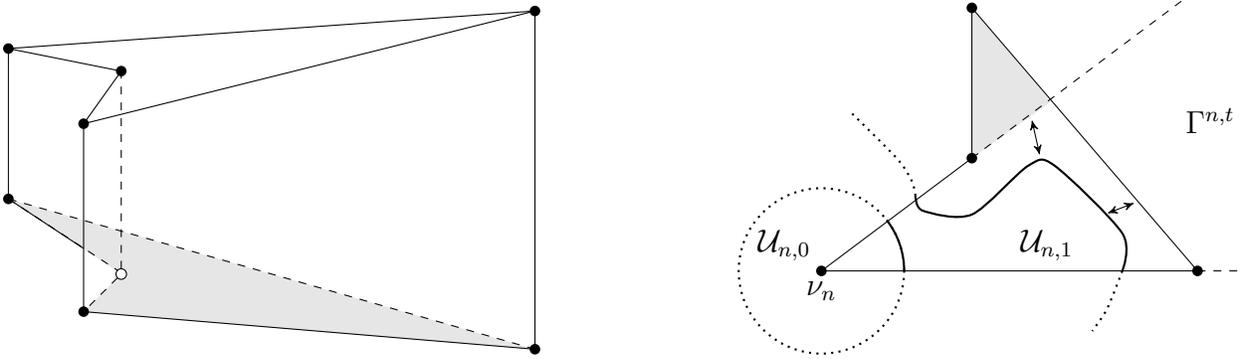 

\begin{remark}
For the sake of simplicity we imposed very strong assumptions on the surfaces under consideration.
This can be done without loss of generality because our proofs require local arguments only. Thus our analysis carries over to a fairly general class of two-dimensional manifolds including important model problems such as, e.g., surfaces of tetrahedra or the boundary of Fichera's corner $\Omega=[-1,1]^3\setminus[0,1]^3$. For this purpose we simply need to introduce additional vertices which subdivide the manifold into quadrilateral patches. We stress that then formally the function spaces we are going to deal with will depend on the concrete choice of these degenerate vertices. However, note that this is a purely theoretical issue which does not diminish the practical applicability in any sense. 
Finally, also the application of diffeomorphisms to the patches $F_i$ would not harm our analysis. Therefore in particular globally smooth manifolds (e.g., spheres) can be treated as well.
\end{remark}

\section{Function spaces of Sobolev type and wavelets}\label{sect:sobolev}
In the first part of this section we give a definition of weighted Sobolev spaces due to~\cite{E92} which can be traced back to~\cite{MP83}.
Afterwards, in \autoref{sect:composite}, we state and discuss some common requirements of wavelet constructions on manifolds that will be needed for our purposes later on.

\subsection{Weighted Sobolev spaces on $\partial\Omega$}\label{sect:weighted}
We begin with weighted Sobolev spaces on the boundary of infinite (tangent) cones. As before the functions of interest are described facewise; see \link{facewise_fkt}.

For compactly supported functions $f_n\colon \partial\Co_n\nach\C$, as well as $k\in\N_0$ and $\rho,\mu\geq 0$, we define $\norm{f_n \sep X^k_{\rho,\mu}(\partial\Co_n)}$ as
\begin{equation*}
	\norm{f_n \sep L_2(\partial\Co_n)} 
		+ \sum_{t=1}^{T_n} \sum_{\substack{\beta=(\beta_r,\beta_\phi)\in\N_0^2\\1\leq\abs{\beta}\leq k}} \norm{ r^{-\mu} (1+r)^{\mu} (q\,r)^{\beta_r} \left( \frac{\partial}{\partial r} \right)^{\beta_r} q^{\beta_\phi-\rho} \left( \frac{\partial}{\partial \phi} \right)^{\beta_\phi} f_{n,t} \sep L_2\!\left(\tilde{\Gamma^{n,t}}\right) },
\end{equation*}
where, as usual, the sum over an empty set is to be interpreted as zero.
Therein $f_{n,t}$ is assumed to be given in polar coordinates $(r,\phi)$ and $q$ is defined by
\begin{equation*}
	q(\phi):=\min\!\left\{ \phi, \gamma_{n,t} - \phi \right\} \in (0,\pi),
\end{equation*}
where $\gamma_{n,t}$ denotes the opening angle of $\tilde{\Gamma^{n,t}}$.
Hence, $r$ and $q$ measure the distance to the face boundary.
For some $C_2>0$ and every $y\in \overline{\tilde{\Gamma^{n,t}}}$ let
\begin{equation*}
	\Delta_{n,t}(y) := \dist{y}{\partial\tilde{\Gamma^{n,t}}}
	\quad \text{and} \quad
	\delta_{n,t}(y) := \min\!\left \{C_2, \Delta_{n,t}(y) \right\}.
\end{equation*}
Then $y=(r\,\cos\phi, r\,\sin\phi)$ yields
\begin{eqnarray}
	\delta_{n,t}(y) 
	&=& \min\!\left\{C_2, r \cdot \sin\!\left( \min\!\left\{ \pi/2, \phi, \gamma_{n,t}-\phi \right\} \right) \right\} \nonumber\\
	&=& \min\!\left\{C_2, r \cdot \sin\!\left( \min\!\left\{ \pi/2, q(\phi) \right\} \right) \right\} \nonumber\\
	&\leq & \min\!\left\{C_2, r \cdot q(\phi) \right\} \nonumber\\
	&\leq & r \cdot q(\phi). \label{eq:bound_delta}
\end{eqnarray}
Following~\cite{E92} weighted Sobolev spaces on the boundary of the cone $\Co_n$ can now be defined as the closure of all continuous, facewise smooth, compactly supported functions on~$\partial\Co_n$ with respect to this norm:
\begin{equation*}
		X^k_{\rho,\mu}(\partial\Co_n)
		:= \overline{C_{0,\mathrm{fw}}^\infty (\partial\Co_n)}^{\, \norm{\cdot \sep X^k_{\rho,\mu}(\partial\Co_n)} }
		\quad \text{and} 
		\quad X^k_{\rho}(\partial\Co_n) := X^k_{\rho,\rho}(\partial\Co_n).
\end{equation*}
In what follows we will exclusively deal with the case $\mu=\rho$ and small values of $k$.
Under these restrictions the name \emph{weigthed Sobolev space} is justified by the following lemma which is proven in \autoref{sect:appendixA}.
\begin{lemma}\label{lem:weighted_Sobolev}
	Let $n\in\{1,\ldots,N\}$, $k\in\N$ and $0\leq \rho \leq k$. Then every $f_n\in X^k_\rho(\partial\Co_n)$ with $\supp f_n \subset \overline{\U_{n,1}}$ satisfies
	\begin{equation*}
		\norm{\delta_{n,t}^{k-\rho} \cdot D^\alpha_y f_{n,t} \sep L_2\!\left(\tilde{\Gamma^{n,t}}\right)}
		\lesssim \norm{f_n \sep X_\rho^k(\partial\Co_n)}
	\end{equation*}
	for all $t\in\{1,\ldots,T_n\}$ and each $\alpha\in\N_0^2$ with $\abs{\alpha}\leq k$.
\end{lemma}

With the help of a special resolution of unity $(\varphi_n)_{n=1}^N$ that localizes functions on $\partial\Omega \subset\bigcup_{n=1}^N \partial\Co_n$ to the boundary of the tangent cones $\Co_n$ of $\Omega$ we define
\begin{equation}\label{def_X}
	\norm{u \sep X^k_\rho(\partial\Omega)} 
	:= \sum_{n=1}^N \norm{\varphi_n u \sep X^k_\rho(\partial\Co_n)}
\end{equation}
for $u\colon \partial \Omega \nach \C$.
Finally, for $k\in\N$ and $0\leq\rho\leq k$, we set
\begin{equation*}
	X^k_{\rho}(\partial\Omega)
		:= \overline{C_{\mathrm{pw}}^\infty (\partial\Omega)}^{\, \norm{\cdot \sep X^k_{\rho}(\partial\Omega)} }.
\end{equation*}

\subsection{Wavelet bases and (unweighted) Sobolev spaces on $\partial\Omega$}\label{sect:composite}
During the past years, wavelets on domains $\Omega \subseteq \R^d$ have become a powerful tool in both pure and applied mathematics. 
More recently, several authors proposed various constructions of wavelet systems extending the idea of multiscale analysis to manifolds based on patchwise descriptions such as \link{partition}; see, e.g., \cite{CTU99,CTU00}, \cite{CM00}, \cite{DS99}, \cite{HaSch04}, \cite{HS06}. Later on in this paper,  bases of these kinds  will be used to define new types  of Besov spaces on $\partial \Omega$. Therefore, in this subsection, we collect some basic properties that will be needed for this purpose.

With the help of the parametric liftings $\kappa_i$, $i=1,\ldots,I$, an inner product for functions $u,v\colon\partial\Omega\nach\C$ can be defined patchwise by
\begin{equation*}
	\distr{u}{v} := \sum_{i=1}^I \distr{u\circ \kappa_i}{v\circ\kappa_i}_{\square},
\end{equation*}
where $\distr{\cdot}{\cdot}_{\square}$ denotes the usual $L_2$-inner product on the square $[0,1]^2$.
Since all $\kappa_i$ are assumed to be sufficiently smooth the norm induced by $\distr{\cdot}{\cdot}$ can be shown to be equivalent to the norm in $L_2(\partial\Omega)$:
\begin{equation}\label{normeq}
	\normmod{\cdot}_0 := \sqrt{\distr{\cdot}{\cdot}} \sim \norm{\cdot \sep L_2(\partial\Omega)},
\end{equation}
see, e.g., Formula (4.5.3) in~\cite{DS99}.

Most of the known wavelet  constructions are based on tensor products of boundary-adapted wavelets/scaling functions (defined on intervals) which are finally lifted to the patches $F_i$ describing the surface $\partial\Omega$.
A wavelet basis $\Psi=(\Psi^{\partial\Omega}, \tilde{\Psi}^{\partial\Omega})$ on $\partial\Omega$
then consists of two collections of functions $\psi^{\partial\Omega}_{j,\xi}$ and $\tilde{\psi}^{\partial\Omega}_{j,\xi}$, respectively, that form ($\distr{\cdot}{\cdot}$-biorthogonal) Riesz bases for $L_2(\partial\Omega)$. 
In particular, every $u\in L_2(\partial\Omega)$ has a unique expansion
\begin{equation}\label{expansion}
	u = P_{j^*-1}(u) + \sum_{j\geq j^*} \sum_{\xi \in {\nabla}_j^{\partial\Omega}} \distr{u}{\tilde{\psi}^{\partial\Omega}_{j,\xi}} \, \psi^{\partial\Omega}_{j,\xi}
\end{equation}
satisfying
\begin{equation*}
		\norm{u\sep L_2(\partial\Omega)} 
		\sim \norm{P_{j^*-1}(u) \sep L_2(\partial\Omega)} + \left( \sum_{j\geq j^*} \sum_{\xi \in {\nabla}_j^{\partial\Omega}} \abs{\distr{u}{\tilde{\psi}^{\partial\Omega}_{j,\xi}}}^2 \right)^{1/2}.
\end{equation*}
Therein $P_{j^*-1}$ denotes the biorthogonal projector that maps $L_2(\partial\Omega)$ onto the finite dimensional span of all generators  on the  coarsest level  $j^*-1$. 

In the sequel, we will require  that the wavelet basis under consideration satisfies all the conditions  collected in the following assumption.

\begin{ass}\label{ass:basis}
\noindent
\begin{itemize}
\item[(I)] As indicated in (\ref{expansion}),  both (the primal and the dual) systems are indexed by their level of resolution $j\geq j^*$, as well as their location (and type) $\xi \in \nabla^{\partial\Omega}$. 
We assume that this collection of grid points on the surface $\partial\Omega$ can be split up according to the levels $j$ and the patches $\overline{F_i}$:
\begin{equation*}
	{\nabla}^{\partial\Omega} 
	= \bigcup_{j=j^*}^\infty {\nabla}_j^{\partial\Omega}
	\quad \text{where, for all } j\geq j^*,
	\quad 
	{\nabla}_j^{\partial\Omega} = \bigcup_{i=1}^I {\nabla}_j^{F_i}
	\quad \text{with}\quad   \#\nabla_j^{F_i} \sim 2^{2j}.
\end{equation*}

\item[(II)]  All dual wavelets are $L_2$-normalized: 
\begin{equation}\label{normalized}
	\norm{\tilde{\psi}^{\partial\Omega}_{j,\xi} \sep L_2(\partial\Omega)} \sim 1
	\qquad \text{for all} \qquad j\geq j^*, \, \xi\in {\nabla}_j^{\partial\Omega}.
\end{equation}

\item[(III)] We assume that all elements $\tilde{\psi}^{\partial\Omega}_{j,\xi} \in \tilde{\Psi}^{\partial\Omega}$ are compactly supported on $\partial\Omega$. Furthermore, we assume that their supports  contain the  corresponding grid point~$\xi$ and satisfy
\begin{equation}\label{supp}
	\abs{\supp \tilde{\psi}^{\partial\Omega}_{j,\xi}} \sim 2^{-2j}
	\qquad \text{for all} \qquad j\geq j^*, \, \xi\in {\nabla}_j^{\partial\Omega}.
\end{equation}

\item[(IV)] Consider the set $\Pi_{\tilde{d}-1}([0,1]^2)$ of polynomials $\P$ on the unit square which have a total degree $\deg\P$ strictly less than $\tilde{d}$.
Then we assume that the  dual system $\tilde{\Psi}^{\partial\Omega}$ satisfies
\begin{equation}\label{vanish}
	\distr{\P}{\tilde{\psi}^{\partial\Omega}_{j,\xi}\circ \kappa_i}_\square = 0
	\qquad
	\text{for all}
	\qquad \P\in \Pi_{\tilde{d}-1}([0,1]^2),
\end{equation}
whenever $\tilde{\psi}^{\partial\Omega}_{j,\xi} \in \tilde{\Psi}^{\partial\Omega}$ is completely supported in the interior of some patch $F_i\subset\partial\Omega$, $i\in\{1,\ldots,I\}$. 
This property is commonly known as \emph{vanishing moment property} of order $\tilde{d}\in\N$.
\item[(V)]  The number of dual wavelets at level $j$ with distance $2^{-j}$ to one of the patch boundaries is of order  $2^{j}$, i.e.,
\begin{equation}\label{number_wavelets1}
\#\!\left\{ \xi \in \nabla^{\partial \Omega}_j \sep 0 < \dist{\supp \tilde{\psi}^{\partial\Omega}_{j,\xi}}{\bigcup_{i=1}^I \partial F_i} \lesssim 2^{-j} \right\} \sim 2^{j} 
\quad \text{for all} \quad j\geq j^*.
\end{equation}
Moreover, for the dual wavelets intersecting one of the patch interfaces we assume that
\begin{equation*}
\#\!\left\{ \xi \in \nabla^{\partial \Omega}_j \sep \supp \tilde{\psi}^{\partial\Omega}_{j,\xi} \cap \bigcup_{i=1}^I \partial F_i \neq \leer \right\} \lesssim 2^{j}
\quad \text{for all} \quad j\geq j^*.
\end{equation*}

\item[(VI)] Every point $x\in\partial\Omega$ is contained in the supports of a uniformly bounded number of dual wavelets at level $j$:
\begin{equation}\label{finite_overlap}
\#\!\left\{ \xi \in \nabla^{\partial \Omega}_j \sep x\in\supp \tilde{\psi}^{\partial\Omega}_{j,\xi} \right\} \lesssim 1
\quad \text{for all} \quad j\geq j^* \quad \text{and each} \quad x\in\partial\Omega.
\end{equation}

\item[(VII)] Finally, we assume that the Sobolev spaces $H^s(\partial\Omega)=W^s(L_2(\partial\Omega))$ in the scale 
\begin{equation*}
	-\frac{1}{2} < s < \min\!\left\{\frac{3}{2}, s_{\partial\Omega}\right\}
\end{equation*}
can be characterized by the decay of wavelet expansion coefficients, that is 

\begin{equation}\label{equiv_sobolev_norm}
		\norm{u\sep H^s(\partial\Omega)} 
		\sim \norm{P_{j^*-1}(u) \sep L_2(\partial\Omega)} + \left( \sum_{j\geq j^*} \sum_{\xi \in {\nabla}_j^{\partial\Omega}} 2^{2sj} \abs{\distr{u}{\tilde{\psi}^{\partial\Omega}_{j,\xi}}}^2 \right)^{1/2}.
\end{equation}

Here the spaces for negative $s$ are defined by duality and $s_{\partial\Omega} \geq 1$ depends on the interior angles between different patches $F_i$ of the manifold under consideration; cf.\ \cite[Section~4.5]{DS99}.
\end{itemize}
\end{ass}

Fortunately, all these  assumptions are satisfied for all  the constructions we mentioned at the beginning of this subsection. In  particular, the \emph{composite wavelet basis} as constructed in~\cite{DS99} is a typical example which will serve as our main reference.
Note that although those wavelets are usually at most continuous across patch interfaces, they are able to capture arbitrary high smoothness in the interior 
by increasing the order of the underlying boundary-adapted wavelets.

\section{Spaces of Besov type and their properties}\label{sect:Besov}
The main objective of this section is to introduce new function spaces on $\partial\Omega$. These spaces are generalized Besov spaces which determine the decay rate of best $n$--term approximation.
Moreover, we investigate a couple of useful properties of these scales of spaces.

\subsection{Besov-type spaces on patchwise smooth manifolds}\label{sect:besov_def}
Besov spaces essentially generalize the concept of Sobolev spaces.
On $\R^d$ they are typically defined using harmonic analysis, finite differences, moduli of smoothness, or interpolation techniques.
Characteristics (embeddings, interpolation results, and approximation properties) of these scales of spaces then require deep proofs within the classical theory of function spaces. Often they are obtained by reducing the assertion of interest to the level of sequences spaces by means of characterizations in terms of building blocks (atoms, local means, quarks, or wavelets).
To mention at least a few references the reader is referred to the monographs \cite{RS96}, \cite{T06}, as well as to the articles \cite{DJP92}, \cite{FJ90}, \cite{KMM07}. This list is clearly not complete.

Besov spaces on manifolds such as boundaries of domains in $\R^d$ can be defined as trace spaces or via pullbacks based on (overlapping) resolutions of unity. In general traces of wavelets are not wavelets anymore, and if we use pullbacks, then wavelet characterizations are naturally limited by the global smoothness of the underlying manifold.
Therefore there seems to exist no approach, suitable for numerical applications, to define higher order (Besov) smoothness of functions defined on surfaces that are only patchwise smooth.

In the following we propose a notion of Besov-type spaces \emph{based on} expansions w.r.t.\ some biorthogonal wavelet Riesz basis $\Psi=(\Psi^{\partial\Omega}, \tilde{\Psi}^{\partial\Omega})$ satisfying the conditions of the previous section which we assume to be given fixed: 

\begin{defi}
A tuple of real parameters $(\alpha, p, q)$ is said to be \emph{admissible} if
\begin{equation}\label{parameter}
	\frac{1}{2} \leq \frac{1}{p} \leq \frac{\alpha}{2} + \frac{1}{2}
	\qquad \text{and} \qquad 
	0 < q 
	\leq \begin{cases}
		2, & \text{if} \qquad 1/p = \alpha/2 + 1/2, \\   
		\infty, & \text{otherwise}.
	\end{cases}
\end{equation}
Given a wavelet basis $\Psi=(\Psi^{\partial\Omega}, \tilde{\Psi}^{\partial\Omega})$ on $\partial\Omega$
and a tuple of admissible parameters $(\alpha, p, q)$
let $B_{\Psi,q}^\alpha(L_p(\partial\Omega))$ denote the collection of all complex-valued functions $u\in L_2(\partial\Omega)$ such that the (quasi-) norm
\begin{equation*}
		\norm{u \sep B_{\Psi,q}^\alpha(L_p(\partial\Omega))}
		:= \norm{P_{j^*-1}(u) \sep L_p(\partial\Omega)} + \left( \sum_{j\geq j^*} 2^{j \left(\alpha+2\left[ \frac{1}{2} - \frac{1}{p} \right] \right)q} \left[ \sum_{\xi \in {\nabla}_j^{\partial\Omega}} \abs{\distr{u}{\tilde{\psi}^{\partial\Omega}_{j,\xi}}}^p \right]^{q/p} \right)^{1/q}
\end{equation*}
is finite. 
If $q=\infty$, then we shall use the usual modification, i.e., we define 
\begin{equation*}
		\norm{u \sep B_{\Psi,\infty}^\alpha(L_p(\partial\Omega))}
		:= \norm{P_{j^*-1}(u) \sep L_p(\partial\Omega)} 
		+ \sup_{j\geq j^*} 2^{j \left(\alpha+2\left[ \frac{1}{2} - \frac{1}{p} \right] \right)} \left[ \sum_{\xi \in {\nabla}_j^{\partial\Omega}} \abs{\distr{u}{\tilde{\psi}^{\partial\Omega}_{j,\xi}}}^p \right]^{1/p}.
\end{equation*}
\end{defi}
\begin{remark}\label{rem:special_besov}
Let us add some comments on this definition:
\noindent
\begin{itemize}
\item[(i)] First of all note that membership of $u$ in $L_2(\partial\Omega)$ implies the unique expansion \link{expansion}. In particular, the projector $P_{j^*-1}$, as well as the wavelet coefficients under consideration, are well-defined.

\item[(ii)] Formally, different bases $\Psi$ might lead to different function spaces even if all remaining parameters $(\alpha,p,q)$ that determine the spaces may coincide.
On the other hand, there are good reasons to conjecture that spaces based on wavelet systems having similar properties (e.g., order of smoothness, vanishing moments) actually coincide. 
This issue is addressed in a separate paper \cite{Wei2014}.
Anyhow, from the application point of view, this does not really matter at all since in any case we are only interested in approximation properties with respect to the \emph{underlying} basis.

\item[(iii)] Observe that the defined expression simplifies when $p=q<\infty$. In this case
\begin{equation*}
		\norm{u \sep B_{\Psi,p}^\alpha(L_p(\partial\Omega))}
		= \norm{P_{j^*-1}(u) \sep L_p(\partial\Omega)} + \left( \sum_{j\geq j^*} 2^{j\alpha p + j(p-2)} \sum_{\xi \in {\nabla}_j^{\partial\Omega}} \abs{\distr{u}{\tilde{\psi}^{\partial\Omega}_{j,\xi}}}^p \right)^{1/p}.
\end{equation*}
Let us briefly discuss the limiting cases for $p$ under this restriction.
For $p=q=2$ the last formula reduces to the right-hand side of \link{equiv_sobolev_norm}. Hence, if $0 \leq\alpha < \min\{3/2, s_{\partial\Omega}\}$, then $B_{\Psi,2}^\alpha(L_2(\partial\Omega))$ equals $H^\alpha(\partial\Omega)$ in the sense of equivalent norms. This particularly covers $L_2(\partial\Omega)$ itself, where $\alpha=0$.
On the other hand, for $p=q:=\tau$ given by  $1/\tau:=\alpha/2+1/2$ the (quasi-) norms further simplify to
\begin{equation}\label{besov_norm}
		\norm{u \sep B_{\Psi,\tau}^\alpha(L_\tau(\partial\Omega))}
		= \norm{P_{j^*-1}(u) \sep L_\tau(\partial\Omega)} + \left( \sum_{j\geq j^*} \sum_{\xi \in {\nabla}_j^{\partial\Omega}} \abs{\distr{u}{\tilde{\psi}^{\partial\Omega}_{j,\xi}}}^\tau \right)^{1/\tau}
\end{equation}
which resembles the so-called adaptivity scale.
\end{itemize}
\end{remark}

In the remainder of this subsection we collect basic properties of the Besov-type spaces introduced above.
To start with, we note that using standard arguments it can be shown that 
all spaces $B_{\Psi,q}^\alpha(L_p(\partial\Omega))$ are quasi-Banach spaces, Banach spaces if and only if $p \geq 1$ and $q\geq 1$, and Hilbert spaces if and only if $p=q=2$.

In \autoref{sect:appendixA} a proof of the following standard embeddings
can be found.
\begin{prop}\label{prop:standard_emb}
For $\gamma \in \R$ let $(\alpha+\gamma,p_0,q_0)$ and $(\alpha,p_1,q_1)$ be admissible parameter tuples.
Then we have the continuous embedding
\begin{equation*}
	B_{\Psi,q_0}^{\alpha+\gamma}(L_{p_0}(\partial\Omega)) \hookrightarrow B_{\Psi,q_1}^{\alpha}(L_{p_1}(\partial\Omega))
\end{equation*}
if and only if one of the following conditions applies
\begin{itemize}
	\item $\gamma > 2 \cdot \max\!\left\{0,\frac{1}{p_0} - \frac{1}{p_1}\right\}$
	\item $\gamma = 2 \cdot \max\!\left\{0,\frac{1}{p_0} - \frac{1}{p_1}\right\}$ \quad and \quad $q_0\leq q_1$.
\end{itemize}
\end{prop}
The subsequent remark discusses some special cases.
\begin{remark}
Let $p_0=q_0<\infty$, as well as $p_1=q_1<\infty$.
\begin{itemize}
\item[(i)] Setting $p_0:=p_1:=2$ and assuming $0 \leq \alpha \leq \alpha+\gamma < \min\{3/2, s_{\partial\Omega}\}$ we see that \autoref{prop:standard_emb} covers the well-known embeddings within the scale of Sobolev (Hilbert) spaces: $H^{\alpha+\gamma}(\partial\Omega)\hookrightarrow H^{\alpha}(\partial\Omega)$.

\item[(ii)] Again let $p_0:=2$. Setting $p_1:=\tau:=(\alpha/2+1/2)^{-1}$ then shows
\begin{equation*}
	H^s(\partial\Omega) \hookrightarrow B_{\Psi,\tau}^{\alpha}(L_{\tau}(\partial\Omega))
	\quad \text{for all} \quad 0 \leq \alpha < s,
\end{equation*}
provided that $0 < s < \min\{3/2, s_{\partial\Omega}\}$ such that \link{equiv_sobolev_norm} applies.
Note that this is optimal in the sense that 
$\alpha \geq s$ is not possible (in this generality) since 
\autoref{prop:standard_emb} shows that there are no embeddings at all if $\gamma < 0$ (or if $\gamma=0$ and $q_0 > q_1$).
However, in \autoref{sect:embeddings} below we will show that 
functions in $H^s(\partial\Omega)$ possess Besov smoothness $\alpha > s$ if they also belong to some weighted Sobolev space.
\item[(iii)] Finally we set $p_0:=\tau_0:=([\alpha+\gamma]/2+1/2)^{-1}$ and $p_1:=\tau_1:=(\alpha/2+1/2)^{-1}$. Then we obtain the natural embedding
\begin{equation*}
	B_{\Psi,\tau_0}^{\alpha+\gamma}(L_{\tau_0}(\partial\Omega)) \hookrightarrow B_{\Psi,\tau_1}^{\alpha}(L_{\tau_1}(\partial\Omega)).
\end{equation*}
along the adaptivity scale. This resembles the classical Sobolev embedding.
\end{itemize} 
\end{remark}

The embeddings stated in \autoref{prop:standard_emb} can be illustrated by a \emph{DeVore-Triebel diagram}; see \autoref{fig:DeVoreTriebel}.
Therein the solid half-lines starting from the point $(1/2,0)$ (which corresponds to the space $L_2(\partial\Omega)$) describe the boundaries of the area of admissible parameters~\link{parameter}. They contain Hilbert spaces such as $H^{s'}(\partial\Omega)=B^{s'}_{\Psi,2}(L_2(\partial\Omega))$ on the left, as well as the adaptivity scale $B^\alpha_{\Psi,\tau}(L_\tau(\partial\Omega))$, with $1/\tau=\alpha/2+1/2$ and $\alpha\geq 0$, on the right.
Moreover, the shaded region refers to all spaces which are embedded into $H^{s'}(\partial\Omega)$, whereas the arrows indicate limiting cases for possible embeddings of the space $B^s_{\Psi,p}(L_p(\partial\Omega))$.
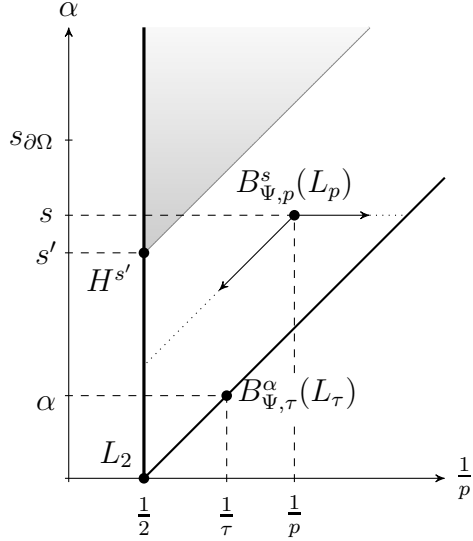
\begin{figure}[H]
	\begin{center}
\begin{tikzpicture}
\draw [->,>=stealth'] (0,-0.05) -- (0,6) node[above] {$\alpha$};
\draw [->,>=stealth'] (-0.05,0) -- (5,0) node[right] {$\frac{1}{p}$};

\draw [gray, thick] (1,3) -- (4,6);
\shade [left color=gray!40, right color=gray!5, shading angle=180] (1,3)--(1,6)--(4,6)--(1,3);

\draw [thick] (1,0) -- (5,4);
\draw [very thick] (1,0) -- (1,6);
\fill (1,0) circle (2pt) node[anchor=south east] {$L_2$};

\draw (-0.3,3) node {$s'$};
\draw [dashed] (-0.05,3) -- (1,3);
\draw (1,0) -- (1,-0.05);
\draw (1,-0.5) node {$\frac{1}{2}$};
\fill (1,3) circle (2pt) node[anchor=north east] {$H^{s'}$};

\draw (-0.05,4.5) node[anchor=east] {$s_{\partial\Omega}$} -- (0.05,4.5);

\draw (-0.3,3.5) node {$s$};
\draw [dashed] (-0.05,3.5) -- (3,3.5);
\fill (3,3.5) circle (2pt) node[anchor=south] {$B^s_{\Psi,p}(L_p)$};
\draw [dashed] (3,-0.05) -- (3,0.7);
\draw [dashed] (3,1.4) -- (3,3.5);
\draw (3,-0.5) node {$\frac{1}{p}$};

\draw [->,>=stealth'] (3,3.5) -- (4,3.5);
\draw [dotted] (4,3.5)--(4.5,3.5);
\draw [->,>=stealth'] (3,3.5) -- (2,2.5);
\draw [dotted] (2,2.5)--(1,1.5);

\draw (-0.3,1) node {$\alpha$};
\draw [dashed] (-0.05,1.1) -- (2.1,1.1);
\fill (2.1,1.1) circle (2pt) node[anchor=west] {$B^\alpha_{\Psi,\tau}(L_\tau)$};
\draw [dashed] (2.1,-0.05) -- (2.1,1.1);
\draw (2.1,-0.5) node {$\frac{1}{\tau}$};
\end{tikzpicture}
	\end{center}
\caption{DeVore-Triebel diagram indicating the area of admissible parameters, as well as standard embeddings stated in \autoref{prop:standard_emb}.}\label{fig:DeVoreTriebel}
\end{figure} 

Finally, the following (complex) interpolation result holds. Its proof, as well as further explanations on the used method, is postponed to \autoref{sect:appendixA}.

\begin{prop}\label{prop:interpol}
	Let $(\alpha_0,p_0,q_0)$ and $(\alpha_1,p_1,q_1)$ be admissible parameter tuples such that $\min\{q_0,q_1\}<\infty$. Then we have (in the sense of equivalent quasi-norms)
\begin{equation*}
	\left[ B_{\Psi,q_0}^{\alpha_0}(L_{p_0}(\partial\Omega)), B_{\Psi,q_1}^{\alpha_1}(L_{p_1}(\partial\Omega))\right]_\Theta
	= B_{\Psi,q_\Theta}^{\alpha_\Theta}(L_{p_\Theta}(\partial\Omega))
	\qquad \text{for all} \qquad 
	\Theta\in(0,1),
\end{equation*}	
	where $(\alpha_\Theta,p_\Theta,q_\Theta)$, given by $\alpha_\Theta:= (1-\Theta)\alpha_0 + \Theta \alpha_1$ , $1/p_\Theta := (1-\Theta)/p_0 + \Theta/p_1$, and $1/q_\Theta := (1-\Theta)/q_0 + \Theta/q_1$, defines an admissible tuple of parameters.
\end{prop}

\subsection{Best $n$--term wavelet approximation}\label{sect:nterm}
Let us recall the abstract definition of the concept of best $n$--term approximation:
\begin{defi}
Let $G$ denote a (quasi-) normed space and let $\B=\{f_1,f_2,\ldots\}$ be some countable subset of $G$. Then
\begin{equation}\label{def_sigma}
	\sigma_n(f;\B, G)
	:= \inf_{i_1,\ldots,i_n \in \N} \inf_{c_1,\ldots,c_n\in\C} \norm{f - \sum_{m=1}^n c_m f_{i_m} \sep G }, \qquad n\in\N,
\end{equation}
defines the error of the \emph{best $n$--term approximation} to some element $f$ w.r.t.\ the \emph{dictionary}~$\B$ in the (quasi-) norm of $G$.
\end{defi}

Since $i_1,\ldots,i_n$ and $c_1,\ldots,c_n$ may depend on $f$ in an arbitrary way this reflects how well we can approximate $f$ using a finite linear combination of elements in $\B$. Obviously these linear combinations form a highly non-linear manifold in the linear space $G$. 
In practice, one is clearly interested in the order of convergence that can be achieved by best $n$--term approximation schemes since this is in a certain sense the best we can expect. Of course, to get a reasonable result, additional information (such as further smoothness properties) of the target function $f$ is needed. This is usually modeled by some additional (quasi-) normed space~$F$.
Hence, if $F$ denotes such a space which is embedded into $G$, then we may study the asymptotic behavior of
\begin{equation*}
	\sigma_n(F;\B,G) := \sup_{\substack{f\in F,\\\norm{f \sep F} \leq 1}} \sigma_n(f;\B, G)
\end{equation*}
because it serves as a benchmark of how well optimal non-linear methods based on the dictionary $\B$ can approximate the embedding $F \hookrightarrow G$.
A prominent class of examples of such methods are algorithms based on wavelet expansions that use adaptive refinement strategies in contrast to linear, non-adaptive methods using uniform grid refinement; see, e.g., \cite[Section~3]{DNS06} for a detailed discussion.

We study the introduced quantities for the case where $F$ and $G$ are Besov-type spaces $B_{\Psi,q}^\alpha(L_p(\partial\Omega))$ and $\B=\Psi^{\partial\Omega}$ is the system of primal wavelets discussed in \autoref{sect:composite}.
The following assertion is proven in \autoref{sect:appendixA}.
\begin{prop}\label{prop:nterm}
	For $\gamma \in \R$ let $(\alpha+\gamma,p_0,q_0)$ and $(\alpha,p_1,q_1)$ be admissible parameter tuples.
	If $\gamma > 2 \cdot \max\!\left\{0,\frac{1}{p_0} - \frac{1}{p_1}\right\}$,
	then
	\begin{equation*}
		\sigma_n \!\left( B_{\Psi,q_0}^{\alpha+\gamma}(L_{p_0}(\partial\Omega));\Psi^{\partial\Omega},B_{\Psi,q_1}^\alpha(L_{p_1}(\partial\Omega)) \right) 
		\sim n^{-\gamma/2}.
	\end{equation*}
	Moreover, if $\gamma = 2 \cdot \max\!\left\{0,\frac{1}{p_0} - \frac{1}{p_1}\right\}$ and $q_0 \leq q_1$, then
	\begin{equation*}
		\sigma_n \!\left( B_{\Psi,q_0}^{\alpha+\gamma}(L_{p_0}(\partial\Omega));\Psi^{\partial\Omega},B_{\Psi,q_1}^\alpha(L_{p_1}(\partial\Omega)) \right) 
		\sim n^{-\min\{\gamma/2,\, 1/q_0-1/q_1\}}. 
	\end{equation*}
\end{prop}

\begin{remark}
Some comments are in order:
\begin{itemize}
\item[(i)] Note that \autoref{prop:nterm} covers all possible embeddings within our scale of Besov-type spaces; see \autoref{prop:standard_emb}.

\item[(ii)] In view of the applications outlined below we are particularly interested in the rate of convergence of best $n$--term approximation of 
\begin{equation*}
	u\in B_{\Psi,\tau}^{s+\gamma}(L_{\tau}(\partial\Omega)),
	\qquad \tau:=(\gamma/2+1/2)^{-1},
\end{equation*}
w.r.t.\ the norm in $H^s(\partial\Omega)$. 
This special case is covered setting $p_1:=q_1:=2$, as well as $p_0:=q_0:=\tau \leq 2$, and
$\alpha:= s \in [0, \max\{3/2, s_{\partial\Omega}\})$.
\end{itemize}
\end{remark}

\begin{remark}
In conclusion, up to now we successfully introduced a new scale of Besov-type spaces of functions on $\partial\Omega$ which behaves similar to well-established scales of function spaces on $\R^d$ and domains, respectively.
In particular, we were able to determine the rate of best $n$--term wavelet approximation in this new scale. 
Note that, in contrast to existing approaches which are, e.g., based on charts, our spaces are well-defined even for arbitrary high smoothness.
To justify the definition completely, it remains to show that these new spaces are reasonably large. 
To this end, we will prove that (under natural conditions) they contain 
functions in classical smoothness spaces which are in addition contained in weighted Sobolev spaces.
As we shall see later on, this particularly covers solutions to boundary operator equations such as the classical double layer equation for the Laplacian.
\end{remark}

\section{Main results}\label{sect:results}
In this section we are going to present our main results on Besov regularity for operator equations defined on patchwise smooth surfaces.
A major step in the proof is based on the continuous embedding 
\begin{equation*}
	B_{\Psi,p}^s(L_p(\partial\Omega)) \cap X_\rho^k(\partial\Omega) 
	\hookrightarrow B_{\Psi,\tau}^\alpha(L_\tau(\partial\Omega))
\end{equation*}
Since it considerably extends the statements in \autoref{prop:standard_emb} it is of interest on its own.
Thus, we formulate it as an individual result (\autoref{thm:embedding}) in the next subsection, whereas the general regularity assertion (\autoref{thm:general_eq}), as well as its application to the double layer potential for the Laplacian (\autoref{thm:doublelayer}), can be found in \autoref{sect:general_eq} and \autoref{sect:doublelayer}, respectively.

Throughout the whole section $\partial\Omega$ denotes the patchwise smooth boundary of some three-dimensional domain $\Omega$ as described in \autoref{sect:boundaries}.
Moreover, we assume to be given a biorthogonal wavelet Riesz basis $\Psi=(\Psi^{\partial\Omega},\tilde{\Psi}^{\partial\Omega})$ on $\partial\Omega$ satisfying the requirements stated in \autoref{ass:basis}, \autoref{sect:composite}.

\subsection{Non-standard embeddings}\label{sect:embeddings}
The embedding theorem reads as follows.
\begin{theorem}\label{thm:embedding}
	Assume $\tilde{d}\in\N$,  $k\in\{1,2,\ldots,\tilde{d}\}$, as well as $\rho\in(0,k)$, and let $(s,p,p)$ be an admissible tuple of parameters with $s>0$. 
	Then all tuples $(\alpha,\tau,\tau)$ satisfying
	\begin{equation}\label{embedding_cond}
		\frac{1}{\tau} = \frac{\alpha}{2}+\frac{1}{2}
		 \quad\, \text{and} \quad\,
		 0 \leq \alpha < 2 \alpha^* 
		 \quad\, \text{with} \quad\,
		\alpha^* := \min\!\left\{\rho, k-\rho, s-\left(\frac{1}{p}-\frac{1}{2} \right)\right\}
	\end{equation}
	are admissible in the sense of \link{parameter} and we have the continuous embedding
	\begin{equation*}
		B_{\Psi,p}^s(L_p(\partial\Omega)) \cap X_\rho^k(\partial\Omega) \hookrightarrow B_{\Psi,\tau}^\alpha(L_\tau(\partial\Omega)).
	\end{equation*}	
	That means, for all $u \in B_{\Psi,p}^s(L_p(\partial\Omega)) \cap X_\rho^k(\partial\Omega)$ it holds
	\begin{equation}\label{est_norm}
		\norm{u \sep B_{\Psi,\tau}^\alpha(L_\tau(\partial\Omega))}
		\lesssim \max\left\{\norm{u \sep B_{\Psi,p}^s(L_p(\partial\Omega))}, \norm{u \sep X_\rho^k(\partial\Omega)}\right\}.
	\end{equation}
\end{theorem}
The stated embedding is illustrated by the left DeVore-Triebel diagram in \autoref{fig:DeVoreTriebel2}.
Therein the solid line corresponds to the target spaces $B_{\Psi,\tau}^\alpha(L_\tau(\partial\Omega))$ satisfying \link{embedding_cond} with $\min\{\rho,k-\rho\}\geq s$.
\begin{figure}[ht]
	\begin{center}
\begin{tikzpicture}
\draw [->,>=stealth'] (0,-0.05) -- (0,6) node[above] {$\alpha$};
\draw [->,>=stealth'] (-0.05,0) -- (6,0) node[right] {$\frac{1}{p}$};

\draw [thin] (1,0) -- (6.5,5.5);
\draw [thin] (1,0) -- (1,6);
\fill (1,0) circle (2pt) node[anchor=south east] {$L_2$};

\draw (-0.3,2.5) node {$s$};
\draw [dashed] (-0.05,2.5) -- (3.5,2.5);
\draw (1,0) -- (1,-0.05);
\draw (1,-0.5) node {$\frac{1}{2}$};
\fill (1,2.5) circle (2pt) node[anchor=south east] {$H^{s}$};
\fill (3.5,2.5) circle (2pt);

\draw (-0.3,5) node {$2s$};
\draw [dashed] (-0.05,5) -- (6,5);

\fill (6,5) circle (2pt);
\draw [->,>=stealth', dotted] (1,2.5) -- (5.8,4.9);

\fill (2,2.5) circle (2pt) node[anchor=north] {$B^s_{\Psi,p}(L_p)$};
\draw [dashed] (2,-0.05) -- (2,2.5);
\draw (2,-0.5) node {$\frac{1}{p}$};

\fill (6,5) circle (2pt);
\draw [->,>=stealth', dotted] (2,2.5) -- (4.8,3.9);

\draw (-0.3,4) node {$2\alpha^*$};
\draw [dashed] (-0.05,4) -- (5,4);
\fill (5,4) circle (2pt) node[anchor=north west] {$B^\alpha_{\Psi,\tau}(L_\tau)$};
\draw [dashed] (5,-0.05) -- (5,4);
\draw (5,-0.5) node {$\frac{1}{\tau}$};
\draw [very thick] (1,0) -- (5,4);
\end{tikzpicture}
		\hfill
\begin{tikzpicture}
\draw [->,>=stealth'] (0,-0.05) -- (0,6) node[above] {$\alpha$};
\draw [->,>=stealth'] (-0.05,0) -- (7,0) node[right] {$\frac{1}{p}$};

\shade [left color=gray!40, right color=gray!5, shading angle=180] (1,1)--(1,6.5)--(6.5,6.5)--(1,1);

\draw [thin] (1,0) -- (6.5,5.5);
\draw [thin] (1,0) -- (1,6);
\draw [thin] (1,1) -- (5.5,5.5);
\fill (1,0) circle (2pt) node[anchor=south east] {$L_2$};

\draw (-0.3,2.5) node {$s$};
\draw [dashed] (-0.05,2.5) -- (1,2.5);
\draw (1,0) -- (1,-0.05);
\draw (1,-0.5) node {$\frac{1}{2}$};
\fill (1,2.5) circle (2pt) node[anchor=south east] {$H^{s}$};

\draw (-0.3,1) node {$s'$};
\draw [dashed] (-0.05,1) -- (1,1);
\fill (1,1) circle (2pt) node[anchor=south east] {$H^{s'}$};

\draw (-0.3,5) node {$2s$};
\draw [dashed] (-0.05,5) -- (6,5);
\draw [dashed] (6,-0.05) -- (6,5);
\draw (6,-0.5) node {$\frac{1}{\tau}$};

\fill (6,5) circle (2pt) node[anchor=north west] {$B^\alpha_{\Psi,\tau}(L_\tau)$};
\draw [thick, dotted] (1,2.5) -- (4,4);
\draw [->,>=stealth', dotted] (4,4) -- (5.8,4.9);

\draw (-0.3,4) node {$\alpha_\theta$};
\draw [dashed] (-0.05,4) -- (4,4);
\fill (4,4) circle (2pt) node[anchor=south east] {$B^{\alpha_\theta}_{\Psi,p_\theta}(L_{p_\theta})$};
\draw [dashed] (4,-0.05) -- (4,4);
\draw (4,-0.5) node {$\frac{1}{p_\theta}$};
\end{tikzpicture}
	\end{center}
\caption{DeVore-Triebel diagrams indicating the results of \autoref{thm:embedding} (left) and \autoref{thm:general_eq} (right). In both cases we assumed $\min\{\rho,k-\rho\}\geq s$ such that we can take $\alpha = 2(s-(1/p-1/2))-\delta$ with arbitrary small $\delta>0$.}\label{fig:DeVoreTriebel2}
\end{figure}
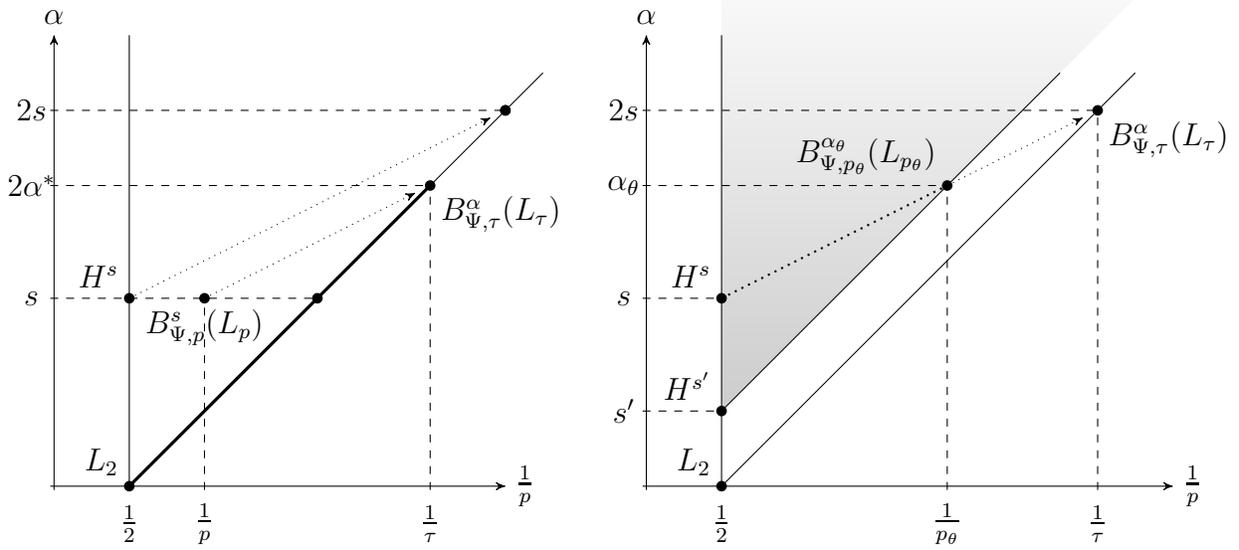

\begin{remark}
From the application point of view the case $p=2$ is of special interest since $B_{\Psi,2}^s(L_2(\partial\Omega)) = H^s(\partial\Omega)$, as long as $s$ is sufficiently small; cf.\ \autoref{rem:special_besov}(iii) in \autoref{sect:besov_def}.
In this case \autoref{thm:embedding} shows that functions~$u$ on the two-dimensional manifold~$\partial\Omega$ which have Sobolev regularity~$s$ possess a Besov regularity $\alpha$ that can be essentially twice as large as $s$, provided that they are sufficiently regular on the smooth parts of the surface (i.e., if $u\in X_\rho^k(\partial\Omega)$ with $\min\{\rho,k-\rho\} \geq s$). This resembles the gain of regularity by a factor of $d/(d-1)$ for functions on bounded Lipschitz domains $\Omega'\subset\R^d$ with additional properties.
For details we refer to~\cite{DD97}, \cite{Han2014}. 
\end{remark}

The proof of \autoref{thm:embedding} which is stated at the end of this subsection is inspired by ideas first given in~\cite{DD97}.
Since it clearly suffices to show \link{est_norm}, we need to estimate wavelet coefficients $\distr{u}{\tilde{\psi}^{\partial\Omega}_{j,\xi}}$ for all $j\geq j^*$ and $\xi \in {\nabla}_j^{\partial\Omega}$; see \link{besov_norm}.
We split our analysis into two parts: 
At first, we bound the contribution of the dual wavelets supported in the vicinity of patch boundaries (\emph{interface} wavelets) in terms of the (quasi-) norm of $u$ in $B_{\Psi,p}^s(L_p(\partial\Omega))$.
Afterwards we take into account all \emph{interior} (dual) wavelets~$\tilde{\psi}^{\partial\Omega}_{j,\xi}$, i.e., those which are completely supported in the interior of some patch $F_i \subset \partial\Omega$. 
Then, using a Whitney-type argument, the respective coefficients will be estimated by the norm of $u$ in the weighted Sobolev space $X^k_\rho(\partial\Omega)$.

We describe the mentioned splitting in detail.
Given $j\geq j^*$ and $i\in\{1,\ldots,I\}$ let ${\nabla}_{j,\mathrm{int}}^{F_i}$ denote the set of all $\xi \in {\nabla}_j^{F_i}$ such that there exist cubes $Q^\square_{j,\xi}\subset (0,1)^2$, as well as balls $B_{j,\xi} \subset F_i$, with the following properties:
\begin{enumerate}[label=(B\arabic{*}), ref=B\arabic{*}]
	\item \label{B1} $\supp \tilde{\psi}^{\partial\Omega}_{j,\xi} \subset \kappa_i(Q^\square_{j,\xi}) \subset B_{j,\xi} \subset F_i$,
	\item \label{B2} $\rad B_{j,\xi} \sim \abs{\supp \tilde{\psi}^{\partial\Omega}_{j,\xi}}^{1/2}$, and
	\item \label{B3} $\dist{B_{j,\xi}}{\partial F_i} > 2^{-j}$.
\end{enumerate}
Here $\rad B$ denotes the radius of the ball $B$.
Setting
\begin{equation*}
	{\nabla}_{j, \mathrm{int}} := \bigcup_{i=1}^I {\nabla}_{j, \mathrm{int}}^{F_i}, 
	\qquad \text{as well as} \qquad 
	{\nabla}_{j, \mathrm{bnd}} := \bigcup_{i=1}^I \left( {\nabla}_j^{F_i} \setminus {\nabla}_{j, \mathrm{int}}^{F_i} \right), 
\end{equation*}
we obviously have 
\begin{equation*}
	{\nabla}_j^{\partial\Omega} = \bigcup_{i=1}^I {\nabla}_j^{F_i} = {\nabla}_{j, \mathrm{int}} \cup {\nabla}_{j,\mathrm{bnd}}
	\qquad \text{for every} \qquad
	j\geq j^*.
\end{equation*}
Moreover, the properties of the underlying wavelet bases imply that the number of dual wavelets at level~$j$ with $\dist{\supp \tilde{\psi}^{\partial\Omega}_{j,\xi}}{\partial F_i} \lesssim 2^{-j}$ scales like $2^j$; see \link{number_wavelets1} in \autoref{ass:basis}. This implies the estimate $\# \left( {\nabla}_j^{F_i} \setminus {\nabla}_{j, \mathrm{int}}^{F_i} \right) \lesssim 2^j$ and hence
\begin{equation}\label{card_bnd}
	\# {\nabla}_{j, \mathrm{bnd}} \lesssim 2^j,
	\qquad \text{whereas} \qquad
	\# {\nabla}_{j, \mathrm{int}} \sim 2^{2j}.
\end{equation}

Now we are well-prepared to bound the wavelet coefficients associated to patch boundaries in terms of the (quasi-) norm of $u$ in the Besov-type space $B_{\Psi,p}^\alpha(L_p(\partial\Omega))$. 
For the proof we refer to \autoref{sect:appendixB}.
\begin{prop}\label{prop:bnd}
	Let $(s,p,p)$ denote an admissible tuple of parameters and assume that
	\begin{equation*}
		\frac{1}{2} \leq \frac{1}{\tau} \leq \frac{1}{p}
		\qquad \text{or} \qquad
		\frac{1}{p} < \frac{1}{\tau} < 1 - \frac{1}{p} + s.
	\end{equation*}	
	Then for all $u\in B_{\Psi,p}^s(L_p(\partial\Omega))$
	\begin{equation*}
		\sum_{j\geq j^*} \sum_{\xi \in {\nabla}_{j,\mathrm{bnd}}} \abs{\distr{u}{\tilde{\psi}^{\partial\Omega}_{j,\xi}}}^\tau 
		\lesssim \norm{u \sep B_{\Psi,p}^s(L_p(\partial\Omega))}^\tau.
	\end{equation*}
\end{prop}

In addition, the contribution of the interior wavelets to the (quasi-) norm of $u$ measured in $B_{\Psi,\tau}^\alpha(L_\tau(\partial\Omega))$, $1/\tau=\alpha/2+1/2$, can be estimated as follows; again a detailed proof is given in \autoref{sect:appendixB}.
\begin{prop}\label{prop:int}
	Let $\tilde{d}\in\N$,
	\begin{equation*}
		k\in\{1,2,\ldots,\tilde{d}\}, \qquad \rho\in(0,k), \qquad \text{and} \qquad
		\frac{1}{2} \leq \frac{1}{\tau} < \frac{1}{2} + \min\{\rho, k-\rho\}.
	\end{equation*}	
	Then for all $u\in X_\rho^k(\partial\Omega)$
	\begin{equation}\label{est_int}
		\sum_{j\geq j^*} \sum_{\xi \in {\nabla}_{j,\mathrm{int}}} \abs{\distr{u}{\tilde{\psi}^{\partial\Omega}_{j,\xi}}}^\tau 
		\lesssim \norm{u \sep X_\rho^k(\partial\Omega)}^\tau.
	\end{equation}
\end{prop}

Combining both the previous bounds the proof of \autoref{thm:embedding} is straightforward:
\begin{proof}[Proof (\autoref{thm:embedding})]
Using the definition of the (quasi-) norm in $B_{\Psi,\tau}^\alpha(L_\tau(\partial\Omega))$ given in \link{besov_norm} together with the introduced splittings of the index sets $\nabla_j^{\partial\Omega}$, and \autoref{lem:finite} from \autoref{sect:appendixA}, we may write
\begin{eqnarray*}
	&&\norm{u \sep B_{\Psi,\tau}^\alpha(L_\tau(\partial\Omega))} \\
	&&\quad \sim \norm{P_{j^*-1}(u) \sep L_\tau(\partial\Omega)} 
		+ \left[ \sum_{j\geq j^*} \sum_{\xi \in {\nabla}_{j, \mathrm{bnd}}} \abs{\distr{u}{\tilde{\psi}^{\partial\Omega}_{j,\xi}}}^\tau \right]^{1/\tau}
		+ \left[ \sum_{j\geq j^*} \sum_{\xi \in {\nabla}_{j, \mathrm{int}}} \abs{\distr{u}{\tilde{\psi}^{\partial\Omega}_{j,\xi}}}^\tau \right]^{1/\tau}.
\end{eqnarray*}
H\"older's inequality in conjunction with \autoref{prop:standard_emb} then implies the estimate
\begin{equation*}
	\norm{P_{j^*-1}(u) \sep L_\tau(\partial\Omega)}
	\lesssim \norm{P_{j^*-1}(u) \sep L_2(\partial\Omega)}
	\leq \norm{u \sep B_{\Psi,2}^0(L_2(\partial\Omega))}
	\lesssim \norm{u \sep B_{\Psi,p}^s(L_p(\partial\Omega))}.
\end{equation*}
Furthermore, \autoref{prop:bnd} and \autoref{prop:int} above allow to bound the remaining terms by $\norm{u \sep B_{\Psi,p}^s(L_p(\partial\Omega))}$ and $\norm{u \sep X_\rho^k(\partial\Omega)}$, respectively.
This proves \link{est_norm}.
\end{proof}

\begin{remark}
Let us conclude this subsection by some final remarks:
\begin{itemize}
\item[(i)] Note that \autoref{prop:bnd} and \autoref{prop:int} can be reformulated as continuity assertions of certain linear projectors $P_\mathrm{int}$ and $P_{\mathrm{bnd}}$ which split up $u$ into a regular part and a singular part near patch interfaces, respectively.
\item[(ii)] We stress that our proof technique can also be used to derive more general embedding theorems where the fine-tuning parameters $q$ of the Besov-type spaces involved not necessarily coincide with their integrability parameters $p$ and $\tau$, respectively.
\end{itemize} 
\end{remark}

\subsection{Besov regularity of general operator equations on manifolds}\label{sect:general_eq}
Given an operator $S$ and a right-hand side $g\colon \partial\Omega\nach \C$ we like to solve the equation
\begin{equation}\label{OperatorS}
	S(u)=g	\quad \text{on} \quad \partial\Omega
\end{equation}
for $u\colon\partial\Omega\nach\C$.
In particular, we are interested in the asymptotic behavior of the error of best $n$--term wavelet approximation to $u$ measured in the norm of $H^{s'}(\partial\Omega)$ for some $s' \geq 0$.

\begin{theorem}\label{thm:general_eq}
	Assume $\tilde{d}\in\N$, $k\in\{1,2,\ldots,\tilde{d}\}$, as well as $\rho\in(0,k)$, and let $(s,p,p)$ be an admissible tuple of  parameters with $s>0$.
	Whenever the solution $u$ to \link{OperatorS} is contained in the intersection of $B_{\Psi,p}^s(L_p(\partial\Omega))$ and $X^k_\rho(\partial\Omega)$, then it also belongs to the Besov-type space $B_{\Psi,\tau}^\alpha(L_\tau(\partial\Omega))$ for all
	tuples $(\alpha,\tau,\tau)$ with
	\begin{equation*}
		\frac{1}{\tau} = \frac{\alpha}{2}+\frac{1}{2}
		 \quad\, \text{and} \quad\,
		 0 \leq \alpha < 2 \alpha^*, 
		 \quad\, \text{where} \quad\,
		\alpha^* = \min\!\left\{\rho, k-\rho, s-\left(\frac{1}{p}-\frac{1}{2} \right)\right\}.
	\end{equation*}	
	Moreover, for every $0 \leq s' < \min\!\left\{ 3/2, s_{\partial\Omega} \right\}$ satisfying
	\begin{equation}\label{bound_s}
		s-s' \geq 2 \left( \frac{1}{p} - \frac{1}{2}  \right)
	\end{equation}
	we have $\sigma_n\!\left( u; \Psi^{\partial\Omega}, H^{s'}(\partial\Omega) \right) \lesssim n^{-\gamma/2}$, as $n\nach\infty$, for all $\gamma < \gamma^*$, where
	\begin{equation*}
		\gamma^* := s-s' + \Theta \cdot (2\alpha^* - s)	\geq 0
		\qquad \text{and} \qquad
		\Theta := 1- \frac{s'}{s-2 \left( 1/p - 1/2 \right)} \in [0,1].
	\end{equation*}
\end{theorem}

\begin{remark}
The second DeVore-Triebel diagram in \autoref{fig:DeVoreTriebel2} illustrates a special case of \autoref{thm:general_eq}. There we have chosen $p=2$ and $0<s'<s<\min\{\rho,k-\rho,3/2,s_{\partial\Omega}\}$, such that particularly $B_{\Psi,p}^s(L_p(\partial\Omega))=H^s(\partial\Omega)$ and $\alpha^*=s$.
The dotted line corresponds to the scale of spaces $B_{\Psi,p_\Theta}^{\alpha_{\Theta}}(L_{p_\Theta}(\partial\Omega))$ which can be reached by complex interpolation of $H^s(\partial\Omega)$ and $B_{\Psi,\tau}^\alpha(L_\tau(\partial\Omega))$; see \autoref{prop:interpol}. 
Interpolation is necessary since only those spaces which belong to the shaded area can be embedded into $H^{s'}(\partial\Omega)$. 
For this special choice of the parameters we obtain that $u\in H^s(\partial\Omega) \cap X^k_\rho(\partial\Omega)$ can be approximated in the norm of $H^{s'}(\partial\Omega)$ at a rate arbitrarily close to $\gamma^*/2 = s-s'$, whereas the rate of convergence for best $n$--term wavelet approximation to an arbitrary function $u\in H^s(\partial\Omega)$ is $(s-s')/2$; see \autoref{prop:nterm}.
Hence, incorporating the additional knowledge about weighted Sobolev regularity (membership in $X^k_\rho(\partial\Omega)$) allows to improve the rate of convergence up to a factor of two. 
\end{remark}

\begin{proof}[Proof (\autoref{thm:general_eq})]
We begin with some obvious observations.
Since membership of $u$ in $B_{\Psi,\tau}^\alpha(L_\tau(\partial\Omega))$ for all $\alpha\in[0,2\alpha^*)$ directly follows from the embedding stated in \autoref{thm:embedding} we only need to show the assertion concerning the rate of best $n$--term wavelet approximation. 
We moreover observe that admissibility of $(s,p,p)$ together with \link{bound_s} implies the range for $\Theta$ using the convention that $0/0:=0$.

Let $\alpha \in [0,2\alpha^*)$ be fixed and assume $s'$ to satisfy the mentioned restrictions.
If $0<s'<s - 2 \left( 1/p - 1/2 \right)$, then $\Theta \in (0,1)$. 
Since $u$ belongs to $B_{\Psi,p}^s(L_p(\partial\Omega)) \cap B_{\Psi,\tau}^\alpha(L_\tau(\partial\Omega))$ it is included in the interpolation space
$\left[ B_{\Psi,p}^s(L_p(\partial\Omega)), B_{\Psi,\tau}^\alpha(L_\tau(\partial\Omega)) \right]_{\Theta}$,
too. 
\autoref{prop:interpol} now yields that the latter space equals $B_{\Psi,p_\Theta}^{\alpha_\Theta}(L_{p_{\Theta}}(\partial\Omega))$ with 
\begin{equation}\label{interpol_para}
	\alpha_\Theta=(1-\Theta)s+\Theta\alpha
	\qquad \text{and} \qquad
	1/p_\Theta=(1-\Theta)/p+\Theta/\tau.
\end{equation}
In the limiting cases $s'=0$ and $s'=s - 2 \left( 1/p - 1/2 \right)$ which correspond to $\Theta=1$ and $\Theta=0$, respectively, the space $B_{\Psi,p_\Theta}^{\alpha_\Theta}(L_{p_{\Theta}}(\partial\Omega))$ simply reduces to $B_{\Psi,\tau}^\alpha(L_\tau(\partial\Omega))$ resp. $B_{\Psi,p}^s(L_p(\partial\Omega))$.
Hence, for every $s'$ under consideration we have
\begin{equation*}
	u \in B_{\Psi,p_\Theta}^{\alpha_\Theta}(L_{p_{\Theta}}(\partial\Omega)),
\end{equation*}
where $\alpha_\Theta$ and $p_\Theta$, given by \link{interpol_para}, depend on the particular choice of $\alpha$.
Straightforward computation shows that
\begin{equation*}
	\frac{1}{p_\Theta} = \frac{\alpha_{\Theta}-s'}{2} + \frac{1}{2}
\end{equation*}
and that $\gamma_{\alpha} := \alpha_{\Theta}-s' = s - s' + \Theta \cdot (\alpha - s)$ 
is non-negative for every $\alpha \in [0,2 \alpha^*)$.

Note that $0\leq s' < \min\{3/2,s_{\partial\Omega}\}$ ensures that we can identify $H^{s'}(\partial\Omega)$ with the Besov-type space $B_{\Psi,2}^{s'}(L_2(\partial\Omega))$.
Moreover, it is well-known that $\sigma_n$ as defined in \link{def_sigma} is homogeneous in the first argument.
Consequently \autoref{prop:nterm} gives
\begin{eqnarray*}
	\sigma_n\!\left( u; \Psi^{\partial\Omega}, H^{s'}(\partial\Omega) \right) 
	 &\leq & \sigma_n\!\left(B_{\Psi,p_\Theta}^{\alpha_\Theta}(L_{p_{\Theta}}(\partial\Omega)); \Psi^{\partial\Omega}, B_{\Psi,2}^{s'}(L_2(\partial\Omega))\right) \cdot \norm{u \sep B_{\Psi,p_\Theta}^{\alpha_\Theta}(L_{p_{\Theta}}(\partial\Omega))} \\
	&\lesssim & n^{-\gamma_{\alpha}/2},
\end{eqnarray*}
as $n\nach\infty$, where $\Psi^{\partial\Omega}$ denotes the system of generators and primal wavelets.
Since $\alpha \in [0,2 \alpha^*)$ was arbitrary we conclude $\sigma_n\!\left( u; \Psi^{\partial\Omega}, H^{s'}(\partial\Omega) \right) \lesssim n^{-\gamma/2}$ for all $\gamma < \gamma^*$ as claimed.
\end{proof}

\subsection{Double layer potential of the Laplacian}\label{sect:doublelayer}
Let $\sigma$ denote the canonical surface measure on the patchwise smooth boundary $\partial\Omega$ of $\Omega\subset\R^3$.
Since this surface is assumed to be Lipschitz, for $\sigma$-a.e.\ $x\in\partial\Omega$ there exists the outward pointing normal vector $\eta(x)$. By $\partial/\partial\eta(x)$ we denote the corresponding \emph{conormal derivative}.
Then the \emph{harmonic double layer potential} on $\overline{\Omega}$ is given by
\begin{equation*}
	v \mapsto K (v) 
	:= \frac{1}{4\pi} \int_{\partial\Omega} v(x) \, \frac{\partial}{\partial \eta(x)} \frac{1}{\abs{x-\cdot\,}_2} \d\sigma(x).
\end{equation*}
We are interested in the solution $u$ to the second kind integral equation
\begin{equation}\label{DoubleLayerProb}
	S_{\mathrm{DL}}(v) 
	:=	\left( \frac{1}{2} \, \Id - K \right)\!(v)
	= g
	\quad \text{on} \quad \partial\Omega
\end{equation}
which naturally arises from the so-called \emph{indirect method} for Dirichlet problems for Laplace's equation in $\Omega$.
Therein $\Id$ denotes the identical mapping and $S_\mathrm{DL}$ is known as \emph{double layer operator}.
For details and further references see, e.g., \cite[Chapter 3.4]{SS11}, as well as \cite{DK87}, \cite{JK81}, \cite{JK95}, \cite{K94}, \cite{V84}.

\begin{theorem}\label{thm:doublelayer}
	Let $s\in(0,1)$, as well as $k\in\N$, and $\rho\in(0,\min\{\rho_0,k\})$ for some $\rho_0 \in (1, 3/2)$ depending on the surface $\partial\Omega$. 
	Moreover let $\alpha$ and $\tau$ be given such that
	\begin{equation*}
		\frac{1}{\tau} = \frac{\alpha}{2}+\frac{1}{2}
		 \quad\, \text{and} \quad\,
		 0 \leq \alpha < 2 \cdot \min\{\rho, k-\rho, s\}
	\end{equation*}	
	and let the Besov-type space $B_{\Psi,\tau}^\alpha(L_\tau(\partial\Omega))$ be constructed with the help of a wavelet basis $\Psi=(\Psi^{\partial\Omega},\tilde{\Psi}^{\partial\Omega})$ possessing vanishing moments of order $\tilde{d} \geq k$. 
Then for every right-hand side $g\in H^s(\partial\Omega)\cap X_\rho^k(\partial\Omega)$ the double layer equation \link{DoubleLayerProb} has a unique solution $u\in B_{\Psi,\tau}^\alpha(L_\tau(\partial\Omega))$.
Furthermore, if $s'\in[0,s]$, then the error of the best $n$--term wavelet approximation to $u$ in the norm of $H^{s'}(\partial\Omega)$ satisfies
\begin{equation*}
	\sigma_n\!\left( u; \Psi^{\partial\Omega}, H^{s'}(\partial\Omega) \right) \lesssim n^{-\gamma/2}
	\qquad \text{for all} \qquad
	\gamma < 2 \cdot \left(1-\frac{s'}{s} \right) \cdot \min\{\rho,k-\rho,s\}.
\end{equation*}
\end{theorem}
\begin{proof}
Thanks to \autoref{thm:general_eq} it suffices to check that, for the stated range of parameters $s$, $k$, and $\rho$, the inverse of the double layer operator $S_\mathrm{DL}$ is a bounded linear operator which maps $H^s(\partial\Omega)\cap X_\rho^k(\partial\Omega)$ onto itself. 
This is covered by \autoref{prop:Elschner} and \autoref{prop:Verchota} below.
The theorem then follows by straightforward calculations.
\end{proof}

Regularity in the weighted Sobolev scale $X_\rho^k(\partial\Omega)$ was established in~\cite[Remark~4.3]{E92}: 
\begin{prop}[Elschner]\label{prop:Elschner}
There exists a constant $\rho_0 \in (1, 3/2)$ depending on the surface $\partial\Omega$ such that the following is true: 
For all $0 \leq \rho < \rho_0$ and every $k\in\N$ with $\rho \leq k$ the bounded linear operator $S_\mathrm{DL} \colon X_\rho^k(\partial\Omega) \nach X_\rho^k(\partial\Omega)$ is invertible.
\end{prop}

Invertibility within the (unweighted) Sobolev scale $H^s(\partial\Omega)$ is known as \emph{Verchota's Theorem} \cite[Theorem 3.3(iii)]{V84}; see also \cite[Remark A.5]{E92}.
\begin{prop}[Verchota]\label{prop:Verchota}
For all $s\in [0,1]$ the linear operator
$S_\mathrm{DL} \colon H^{s}(\partial\Omega) \nach H^{s}(\partial\Omega)$ is boundedly invertible.
\end{prop}

Note that in the notation of \cite{V84} $K$ is replaced by $-K$ and that the operator $K$ considered in \cite{E92} differs from our notation by a factor of $1/2$. Nevertheless, the whole analysis carries over.

\begin{remark}
Verchota's theorem holds true for general Lipschitz surfaces. In \cite[Theorem~3.2.3]{SS11} it is claimed that in the more restrictive situation of patchwise smooth boundaries there exists a constant $s_0 \in [1, s_{\partial\Omega})$ depending on the surface $\partial\Omega$ such that the following is true for all $s\in [1/2, s_0)$: If $g\in H^s(\partial\Omega)$, then every solution $u\in H^{1/2}(\partial\Omega)$ to \link{DoubleLayerProb} is contained in $H^{s}(\partial\Omega)$.
This resembles an assertion stated in~\cite{C88}. 
Although this would allow to extend \autoref{thm:doublelayer} to $s<s_0$ we restricted ourselves to $s<1$ 
since no sound proof for this claim was found in the literature.
\end{remark}

\begin{remark}
To summarize the obtained results we can state that the usage of adaptive wavelet algorithms for operator equations that satisfy certain (weighted) Sobolev regularity assumptions on patchwise smooth manifolds is indeed justified. The reason is that (under natural conditions on the underlying wavelet basis) solutions to such equations will have a Besov regularity which is significantly higher than their Sobolev regularity. As already explained in the introduction, this indicates that adaptive wavelet--based strategies can outperform well-established uniform approximation schemes.
\end{remark}

\begin{appendix}
\section{Appendix}\label{sect:appendixA}
Here we provide some standard assertions, as well as the proofs of 
the technical lemmata and propositions needed above.

For the reader's convenience we start with a well-known result from functional analysis.
\begin{lemma}\label{lem:finite}
	For $\tau > 0$ and $N\in\N$ the (quasi-) norms $\norm{\cdot \sep \l_\tau^N}$ and $\norm{\cdot \sep \l_1^N}$
	are equivalent, i.e.,
	\begin{equation*}
		\left( \sum_{n=1}^N \abs{x_n} \right)^\tau 
		\sim \sum_{n=1}^N \abs{x_n}^\tau
		\qquad \text{for all} \qquad  (x_n)_{n=1}^N \in \C^N.
	\end{equation*}
\end{lemma}

The next classical assertion from approximation theory gives an estimate of the local error of the best-approximation to a given function $f$ in $L_p$ by polynomials of finite degree. 
The proof  of this result (based on Taylor polynomials) is standard. Variants, generalizations, and further remarks are stated in~\cite[Section~6.1]{D98}.
\begin{prop}[Whitney estimate]\label{prop:whitney}
	Let $Q$ be some cube in $\R^d$ with edges parallel to the coordinate axes. 
	Then for all $p\in[1, \infty]$, $k\in\N$, and $\gamma,\zeta$ with
	\begin{equation*}
		\zeta = \frac{k}{d} - \frac{1}{\gamma} + \frac{1}{p}
	\end{equation*}
	we have
	\begin{equation*}
		\inf_{\P\in \Pi_{k-1}(Q)} \norm{f-\P \sep L_p(Q)} 
		\lesssim \abs{Q}^\zeta \, \abs{f}_{W^k(L_\gamma(Q))}
	\end{equation*}
	whenever the right-hand side is finite.
	Here $\Pi_{k-1}(Q)$ denotes the set of all polynomials $\P$ on $Q$ with total degree $\deg \P < k$.
\end{prop}

It remains to prove some lemmata and propositions.
We start with the proof of \autoref{lem:weighted_Sobolev} in \autoref{sect:weighted}.
\begin{proof}[Proof (\autoref{lem:weighted_Sobolev})]
Let $\alpha\in\N_0^2$. 
If $\alpha=(0,0)$, then we use $\delta_{n,t}\leq C_2$ and $k\geq \rho$ to conclude
\begin{equation*}
	\norm{\delta_{n,t}^{k-\rho} \cdot f_{n,t} \sep L_2\!\left(\tilde{\Gamma^{n,t}}\right)}
	\lesssim \sum_{t'=1}^{T_n} \norm{f_{n,t'} \sep L_2\!\left(\tilde{\Gamma^{n,t'}}\right)}
	= \norm{f_{n} \sep L_2(\partial\Co_n)}
	\leq \norm{f_n \sep X_\rho^k(\partial\Co_n)}.
\end{equation*}

For the remaining cases where $1 \leq \abs{\alpha}\leq k$, the derivatives $D^{\alpha}_y=\partial^{\abs{\alpha}}/(\partial y_1^{\alpha_1}\partial y_2^{\alpha_2})$ with respect to $y=(y_1,y_2)\in \tilde{\Gamma^{n,t}}$ can be expressed in terms of polar coordinates $(r,\phi)$ using
\begin{equation*}
	\frac{\partial}{\partial y_1} = \cos\phi \, \frac{\partial}{\partial r} - \frac{\sin\phi}{r} \, \frac{\partial}{\partial\phi}
	\qquad \text{and} \qquad
	\frac{\partial}{\partial y_2} = \sin\phi \, \frac{\partial}{\partial r} + \frac{\cos\phi}{r} \, \frac{\partial}{\partial\phi}
\end{equation*}
together with the chain and product rule.
This results in differential operators
\begin{equation*}
	\sum_{\substack{\beta=(\beta_r,\beta_\phi)\in\N_0^2\\1\leq \abs{\beta}\leq \abs{\alpha}}} c_\beta(\phi) \, r^{\beta_r-\abs{\alpha}} \, \frac{\partial^{\abs{\beta}}}{\partial r^{\beta_r}\partial \phi^{\beta_\phi}},
\end{equation*}
where the $c_\beta$ are multiples of trigonometric functions and thus bounded.
Hence, we may apply the triangle inequality in $L_2$, omit $c_\beta(\phi)$, and estimate $\delta_{n,t}$ by $r\cdot q(\phi)$ due to \link{eq:bound_delta}.
In order to form the building blocks of the norm in $X^k_\rho(\partial\Co_n)$ we can plug in arbitrary (non-negative) powers of $(1+r)$ and bound $q(\phi)$ by $\pi$ (if necessary). 
The claim then follows from the fact that for $\abs{\alpha}<k$ the supernumerary factor $r^{k-\abs{\alpha}}$ can be neglected as well, because (due to the assumption on the support of $f_n$) the domain of integration is actually bounded uniformly in~$f_n$.
\end{proof}
 
We continue by proving the claimed assertions concerning standard embeddings, interpolation and best $n$--term approximation within the scale $B_{\Psi,q}^\alpha(L_p(\partial\Omega))$ of Besov-type spaces introduced in \autoref{sect:Besov}.
\begin{proof}[Proof (\autoref{prop:standard_emb})]
\emph{Step 1.}
Let $(\alpha,p,q)$ be an admissible tuple of parameters; see \link{parameter}.
Then, due to the properties of the underlying wavelet basis (cf.~\autoref{ass:basis} in \autoref{sect:composite}), we may rewrite the (quasi-) norm in $B_{\Psi,q}^\alpha(L_p(\partial\Omega))$ in terms of the sequence space (quasi-) norm
\begin{equation}\label{seq_space}
	\norm{ \bm{a} \sep b^\alpha_{p,q}(\nabla)} 
	:= \left( \sum_{j=0}^{\infty} 2^{j \left(\alpha+2\left[ \frac{1}{2} - \frac{1}{p} \right] \right)q} \left[ \sum_{\lambda \in {\nabla}_j} \abs{ a_{j,\lambda} }^p \right]^{q/p} \right)^{1/q}
\end{equation}
in $b^\alpha_{p,q}(\nabla) := \left\{ \bm{a}=(a_{j,\lambda})_{j\in\N_0, \lambda\in \nabla_j}\subset \C \sep \norm{ \bm{a} \sep b^\alpha_{p,q}(\nabla) } < \infty \right\}$; see, e.g.,~\cite[Definition~3]{DNS06}.
Therein $\nabla:=(\nabla_j)_{j\in\N_0}$ denotes a sequence of finite subsets of the set $\{1,2,3\}\times \Z^2$ such that
\begin{equation*}
	\#\nabla_j \sim 2^{2j}.
\end{equation*}
To do so we only need to choose bijections that map $\xi\in\nabla_{j'}^{\partial\Omega}$ onto $\lambda \in \nabla_{j}$, where $j'=j+j_0$ for some fixed shift $j_0\in\N_0$ depending on $j^*$. Note that the first summand in the Besov (quasi-) norm that refers to the projector $P_{j*-1}$ can be incorporated as well since also this term can be expanded into a sum of (finitely many) inner products. 

Consequently we obtain an isomorphism $\I$ that identifies $u\in B_{\Psi,q}^\alpha(L_p(\partial\Omega))$ with $\bm{a}=\bm{a}(u)\in b^\alpha_{p,q}(\nabla)$, where the latter corresponds to the sequence of all coefficients $\distr{u}{\tilde{\psi}_{j,\xi}^{\partial\Omega}}$. Hence,
\begin{equation*}
	\norm{ \bm{a}(\cdot) \sep b^\alpha_{p,q}(\nabla)} 
	\sim \norm{ \cdot \sep B_{\Psi,q}^\alpha(L_p(\partial\Omega))}.
\end{equation*}

\emph{Step 2.}
Now the desired assertion is equivalent to the corresponding embedding result for sequence spaces $b^\alpha_{p,q}(\nabla)$.
In the case $\gamma \geq 0$ this can be found (for an even wider range of parameters), e.g., in \cite[Lemma~4]{DNS06}. On the other hand, if $\gamma < 0$, then the sequence $\bm{a^*}:=(a^*_{j,\lambda})_{j\in\N_0, \lambda\in \nabla_j}$ defined by $a^*_{j,\lambda}:= 2^{-j(1+\alpha+\gamma/2)}$ for every $\lambda\in\nabla_j$, $j\in\N_0$, belongs to $b_{p_0,q_0}^{\alpha+\gamma}(\nabla) \setminus b_{p_1,q_1}^{\alpha}(\nabla)$. This contradicts $b_{p_0,q_0}^{\alpha+\gamma}(\nabla) \hookrightarrow b_{p_1,q_1}^{\alpha}(\nabla)$.
\end{proof} 

Before we give the proof of the interpolation result stated in \autoref{prop:interpol} let us add some remarks on basic aspects of interpolation theory.
Given some \emph{interpolation couple} of complex (quasi-) Banach spaces $\{A_1,A_2\}$, the application of an \emph{interpolation functor} $\mathfrak{I}$ results in another (quasi-) Banach space $A:=\mathfrak{I}(A_1,A_2)$ for which the following continuous embeddings hold
\begin{equation*}
	A_1 \cap A_2 \hookrightarrow A=\mathfrak{I}(A_1,A_2) \hookrightarrow A_1+A_2
\end{equation*}
and which possesses the so-called \emph{interpolation property}.
That is, every linear operator $T\colon A_1+A_2 \nach B_1+B_2$ which is continuous considered as a mapping $T\colon A_\ell \nach B_\ell$, $\ell\in\{1,2\}$, will act as a bounded linear map between the \emph{interpolation spaces} $\mathfrak{I}(A_1,A_2)$ and $\mathfrak{I}(B_1,B_2)$.
Classical examples for interpolation functors are the well-known real and complex interpolation methods for Banach spaces due to Lions/Peetre and Calder\'{o}n, respectively; see, e.g., \cite{BL76}.
Since we want to deal also with quasi-Banach spaces we make use of a generalization of the complex interpolation method established in~\cite{KMM07}, with~\cite{MM00} as a forerunner. The corresponding functor will be denoted by $\mathfrak{I}
(\cdot,\cdot):=[\cdot,\cdot]_\Theta$ for some $\Theta$ in~$(0,1)$.

\begin{proof}[Proof (\autoref{prop:interpol})]
\emph{Step 1.}
Recall that for every admissible set of parameters $(\alpha,p,q)$ the Besov-type space $B_{\Psi,q}^{\alpha}(L_{p}(\partial\Omega))$ can be identified with some sequence space $b^\alpha_{p,q}(\nabla)$ by means of the (universal) isomorphism $\I\colon u\mapsto \bm{a}(u)$ given in Step 1 of the proof of \autoref{prop:standard_emb}. 
Thanks to the continuity of $\I$ and $\I^{-1}$, as well as the interpolation property, it is enough to show that
\begin{equation}\label{interpol_claim}
	\left[ b_{p_0,q_0}^{\alpha_0}(\nabla), b_{p_1,q_1}^{\alpha_1}(\nabla) \right]_\Theta
	= b_{p_\Theta,q_\Theta}^{\alpha_\Theta}(\nabla).
\end{equation}
Note that admissibility of the parameter tuple $(\alpha_\Theta,p_\Theta,q_\Theta)$ follows from the fact that the domain of admissibility~\link{parameter} is convex.

\emph{Step 2.}
To prove \link{interpol_claim} we consider the canonical extension of sequences $\bm{a} \in b_{p,q}^{\alpha}(\nabla)$ to $\tilde{\bm{a}}:=\Ext\bm{a} \in b_{p,q}^\alpha(\tilde{\nabla})$, where the latter space is given by \link{seq_space} using the collection of index sets $\tilde{\nabla}:=(\tilde{\nabla}_j)_{j\in\N_0}$ defined by
\begin{equation*}
	\tilde{\nabla}_j:=\begin{cases}
		\{1\}\times \Z^2, & j=0,\\
		\{1,2,3\} \times \Z^2, & j\in\N,
	\end{cases}
\end{equation*}
instead of $\nabla$.
That is, we consider the (universal) bounded linear operator $\Ext \colon b_{p,q}^{\alpha}(\nabla) \nach b_{p,q}^\alpha(\tilde{\nabla})$, mapping $\bm{a}=(a_{j,\lambda})_{j\in\N_0,\lambda\in\nabla_j}$ to the sequence $\Ext \bm{a}$, given by
\begin{equation*}
	(\Ext \bm{a})_{j,\lambda} := \begin{cases}
		a_{j,\lambda}, & j\in\N_0 \text{ and } \lambda\in\nabla_j,\\
		0, & j\in\N_0 \text{ and } \lambda\in\tilde{\nabla}_j\setminus\nabla_j.
	\end{cases}
\end{equation*}
Since $\nabla_j \subset \tilde{\nabla}_j$ for all $j\in\N_0$, also the restriction
\begin{equation*}
	\Res \colon b_{p,q}^\alpha(\tilde{\nabla}) \nach b_{p,q}^\alpha(\nabla),
	\qquad \tilde{\bm{a}} \mapsto \Res \bm{\tilde{a}}:=(\tilde{a}_{j,\lambda})_{j\in\N_0,\lambda\in \nabla_j},
\end{equation*}
is a well-defined, linear and bounded mapping which is independent of the (admissible) parameter tuple $(\alpha,p,q)$.

We follow the lines of~\cite[Theorem 1.110]{T06} and define the shortcuts
\begin{equation*}
	b_\Theta(\nabla) := \left[ b_{p_0,q_0}^{\alpha_0}(\nabla), b_{p_1,q_1}^{\alpha_1}(\nabla) \right]_\Theta
	\quad
	\text{and}
	\quad 
	b_\Theta(\tilde{\nabla}) := \left[ b_{p_0,q_0}^{\alpha_0}(\tilde{\nabla}), b_{p_1,q_1}^{\alpha_1}(\tilde{\nabla}) \right]_\Theta.
\end{equation*}
That is, we need to show that $b_\Theta(\nabla)=b_{p_\Theta,q_\Theta}^{\alpha_\Theta}(\nabla)$. 
Now suppose we had already proven that $b_\Theta(\tilde{\nabla})$ equals $b_{p_\Theta,q_\Theta}^{\alpha_\Theta}(\tilde{\nabla})$ in the sense of equivalent (quasi-) norms.
Then the equality $\bm{a}=\Res(\Ext \bm{a})$ for all $\bm{a}\in b_{p_0,q_0}^{\alpha_0}(\nabla) + b_{p_1,q_1}^{\alpha_1}(\nabla)$ gives
\begin{eqnarray*}
	\norm{\bm{a} \sep b_{p_\Theta,q_\Theta}^{\alpha_\Theta}(\nabla)}
	\lesssim  \norm{\Ext\bm{a} \sep b_{p_\Theta,q_\Theta}^{\alpha_\Theta}(\tilde{\nabla})} 
	\sim  \norm{ \Ext\bm{a} \sep  b_\Theta(\tilde{\nabla})}
	\lesssim \norm{ \bm{a} \sep  b_\Theta(\nabla)},
\end{eqnarray*}
i.e., $b_\Theta(\nabla) \hookrightarrow b_{p_\Theta,q_\Theta}^{\alpha_\Theta}(\nabla)$, 
since the interpolation property implies that $\Ext \colon b_\Theta(\nabla) \nach b_\Theta(\tilde{\nabla})$ is continuous.
Conversely, the interpolation property applied for $\Res$ yields 
\begin{eqnarray*}
	\norm{\bm{a} \sep b_\Theta(\nabla)}
	\lesssim  \norm{\Ext\bm{a} \sep b_\Theta(\tilde{\nabla})} 
	\sim  \norm{ \Ext\bm{a} \sep  b_{p_\Theta,q_\Theta}^{\alpha_\Theta}(\tilde{\nabla}) }
	\lesssim \norm{ \bm{a} \sep   b_{p_\Theta,q_\Theta}^{\alpha_\Theta}(\nabla)}.
\end{eqnarray*}
Therefore $b_{p_\Theta,q_\Theta}^{\alpha_\Theta}(\nabla) \hookrightarrow b_\Theta(\nabla)$ which shows that 
$b_\Theta(\nabla)=\left[ b_{p_0,q_0}^{\alpha_0}(\nabla), b_{p_1,q_1}^{\alpha_1}(\nabla) \right]_\Theta$ actually equals $b_{p_\Theta,q_\Theta}^{\alpha_\Theta}(\nabla)$ as claimed in \link{interpol_claim}.

\emph{Step 3.}
It remains to prove that our assumption $b_\Theta(\tilde{\nabla})=b_{p_\Theta,q_\Theta}^{\alpha_\Theta}(\tilde{\nabla})$ holds true for the range of parameters stated in the proposition.
Combining the isomorphism constructed in \cite[Theorem 1.64]{T06} with the arguments given in Step 1 we see that this assertion reduces to an interpolation result for Besov spaces $B^\alpha_{p,q}(\R^2)$ defined (e.g., using harmonic analysis) on the whole of $\R^2$. The needed result can be found, e.g., in \cite[Theorem~9.1]{KMM07}.
\end{proof}

\begin{remark}
Observe that the proof of \autoref{prop:interpol} does not depend on the concrete interpolation functor $\mathfrak{I}(\cdot,\cdot)=[\cdot,\cdot]_\Theta$.
Since there is no restriction on the parameters $(\alpha,p,q)$ in \cite[Theorem 1.64]{T06} every interpolation result known for Besov spaces $B^\alpha_{p,q}(\R^2)$ remains valid for the scale $B_{\Psi,q}^{\alpha}(L_{p}(\partial\Omega))$, provided that the parameter tuples are admissible. 
\end{remark}

\begin{proof}[Proof (\autoref{prop:nterm})]
Assume $\gamma \in \R$ and let $(\alpha+\gamma,p_0,q_0)$ and $(\alpha,p_1,q_1)$ be admissible in the sense of \link{parameter}.
Using the isomorphism in Step 1 of the proof of \autoref{prop:standard_emb} it is enough to prove the assertion at the level of sequence spaces $b^\alpha_{p,q}(\nabla)$ since it is easily seen that
\begin{equation*}
	\sigma_n \!\left(B_{\Psi,q_0}^{\alpha+\gamma}(L_{p_0}(\partial\Omega)); \Psi^{\partial\Omega}, B_{\Psi,q_1}^\alpha(L_{p_1}(\partial\Omega)) \right)
	\sim \sigma_n \!\left( b^{\alpha+\gamma}_{p_0,q_0}(\nabla); \B_\mathrm{seq}, b^\alpha_{p_1,q_1}(\nabla) \right),
\end{equation*}
where $\B_\mathrm{seq} := \{ e_{j,\lambda} \sep j\in\N_0, \lambda \in \nabla_j \}$ denotes the canonical basis in $b^0_{2,2}(\nabla)$.

If $\gamma > 2 \cdot \max\{0,1/p_0 - 1/p_1\}$, then the claimed result is covered by \cite[Theorem~7]{DNS06}.

The upper bound for the case $\gamma=0$ simply follows from the continuity of the corresponding embedding, because
for all $\bm{a}$ in the unit ball of $b_{p_0,q_0}^\alpha(\nabla)$ and every $n\in\N$ it is
\begin{equation*}
	\sigma_n \! \left( \bm{a}; \B_{\mathrm{seq}}, b_{p_1,q_1}^\alpha(\nabla) \right) \leq \norm{ \bm{a} \sep b_{p_1,q_1}^\alpha(\nabla)}
	\lesssim \norm{ \bm{a} \sep b_{p_0,q_0}^\alpha(\nabla)},
\end{equation*}
provided that $1/p_0 \leq 1/p_1$ and $q_0\leq q_1$; see \cite[Lemma~4]{DNS06}.
The corresponding lower bound for this case is implied by the following example.
Given $n\in\N$ choose $j^*\in\N$ such that $n \leq \# \nabla_{j^*} / 2$ and consider the sequence
$\bm{a^*} := (a^*_{j,\lambda})_{j\in\N_0,\lambda\in\nabla_j}$ defined by 
\begin{equation*}
	a^*_{j,\lambda}
	:= \begin{cases}
		2^{-j^*(\alpha+1)}, & \text{if} \quad j=j^*, \lambda \in \nabla_{j^*},\\
		0, & \text{otherwise}.
	\end{cases}
\end{equation*}
Then it is easy to check that 
$\norm{\bm{a^*} \sep b_{p_0,q_0}^\alpha(\nabla)} = (2^{-2j^*} \# \nabla_{j^*})^{1/p_0}$.
Due to the special structure of this sequence and of the spaces $b_{p,q}^\alpha(\nabla)$, every partial sum that consists of exactly $n$ non-trivial terms yields an optimal approximation $A_n(\bm{a^*})$.
Note that $\bm{a^*}$ and $A_n(\bm{a^*})$ are finitely supported and thus contained in every $b_{p,q}^\alpha(\nabla)$.
Moreover, we see that $\norm{\bm{a^*}-A_n(\bm{a^*}) \sep b_{p_1,q_1}^\alpha(\nabla)} = (2^{-2j^*} (\# \nabla_{j^*}-n))^{1/p_1}$. 
Now the estimates $\# \nabla_{j^*}-n \geq \# \nabla_{j^*}/2$ and $\# \nabla_{j^*} \sim 2^{2j^*}$ imply
the desired result
\begin{equation*}
	\sigma_n \!\left( b^{\alpha}_{p_0,q_0}(\nabla); \B_\mathrm{seq}, b^\alpha_{p_1,q_1}(\nabla) \right)
	\geq \frac{\norm{\bm{a^*}-A_n(\bm{a^*}) \sep b_{p_1,q_1}^\alpha(\nabla)}}{\norm{\bm{a^*} \sep b_{p_0,q_0}^\alpha(\nabla)}}
	\geq \frac{1}{2^{1/p_1}} \left (\frac{\#\nabla_{j^*}}{2^{2j^*}} \right)^{1/p_1-1/p_0} \sim 1.
\end{equation*}

We are left with the case $\gamma = 2 \cdot (1/p_0 - 1/p_1) > 0$ and $q_0\leq q_1$.
Here we rely on the estimate
\begin{equation*}
	\sigma_n \!\left( b^{\alpha_0}_{p_0,q_0}(\nabla); \B_\mathrm{seq}, b^{\alpha_1}_{p_1,q_1}(\nabla) \right) 
	\sim n^{-r}
\end{equation*}
which holds for all $0<p_0<p_1\leq \infty$, $0<q_0\leq q_1\leq \infty$, and $\alpha_0,\alpha_1\in\R$ such that
\begin{equation*}
	\alpha_0 - \frac{2}{p_0} = \alpha_1 - \frac{2}{p_1}
	\qquad \text{and} \qquad
	r:= \min\left\{ \frac{1}{p_0}-\frac{1}{p_1}, \frac{1}{q_0}-\frac{1}{q_1} \right\}.
\end{equation*}
Its proof can be derived by adapting the arguments used to show \cite[Theorem~4]{HanSic2011}.
The upper bound for the case where both the entries of the latter minimum coincide can be traced back to \cite{Kyr2001}; see \cite{HanSic2011} for details.
\end{proof}

\section{Appendix}\label{sect:appendixB}
This final section contains the derivation of \autoref{prop:bnd} and \autoref{prop:int} from \autoref{sect:embeddings}, which are essential ingredients in the proof of our main \autoref{thm:embedding}.
\begin{proof}[Proof (\autoref{prop:bnd})]
\emph{Step 1.}
To begin with, consider the case $1/2 \leq 1/\tau \leq 1/p$ and define $\alpha:=2(1/\tau-1/2)$, i.e., $1/\tau = \alpha/2 + 1/2$.
Since $(s,p,p)$ satisfies \link{parameter} we see that 
\begin{equation*}
	\gamma:= s-\alpha \geq 2 \cdot \left( \frac{1}{p} - \frac{1}{\tau}\right) \geq 0.
\end{equation*}
Hence, we may use the embedding result from \autoref{prop:standard_emb} with
$p_0:=q_0:=p$ and $p_1:=q_1:=\tau$ to conclude
\begin{eqnarray*}
	\sum_{j\geq j^*} \sum_{\xi \in {\nabla}_{j,\mathrm{bnd}}} \abs{\distr{u}{\tilde{\psi}^{\partial\Omega}_{j,\xi}}}^\tau 
	&\leq & \norm{u \sep B_{\Psi,\tau}^\alpha(L_\tau(\partial\Omega))}^\tau \\
	&\lesssim & \norm{u \sep B_{\Psi,p}^{\alpha+\gamma}(L_p(\partial\Omega))}^\tau
	= \norm{u \sep B_{\Psi,p}^{s}(L_p(\partial\Omega))}^\tau,
\end{eqnarray*}
where we extended the inner summation from $\nabla_{j,\mathrm{bnd}}$ to $\nabla_j^{\partial\Omega}$ to obtain the first estimate.

\emph{Step 2.}
Now assume $1/p < 1/\tau < \infty$ and let $j\geq j^*$ be fixed.
Then it follows from H\"older's inequality with conjugate exponents $p/\tau$ and $p/(p-\tau)$ that
\begin{eqnarray*}
	\sum_{\xi \in {\nabla}_{j,\mathrm{bnd}}} \abs{\distr{u}{\tilde{\psi}^{\partial\Omega}_{j,\xi}}}^\tau 
	&\leq & \left( \sum_{\xi \in {\nabla}_{j,\mathrm{bnd}}} \abs{\distr{u}{\tilde{\psi}^{\partial\Omega}_{j,\xi}}}^p \right)^{\tau/p} \left(\# {\nabla}_{j,\mathrm{bnd}} \right)^{(p-\tau)/p}\\
	&\lesssim & \left( \sum_{\xi \in {\nabla}_{j,\mathrm{bnd}}} 2^{j(sp + p-2)}\abs{\distr{u}{\tilde{\psi}^{\partial\Omega}_{j,\xi}}}^p \right)^{\tau/p} 2^{ j(p-\tau)/p-j(sp+p-2)\tau/p},
\end{eqnarray*}
since the cardinality of ${\nabla}_{j,\mathrm{bnd}}$ scales like $2^j$; see \link{card_bnd}.
Next we take the sum over all $j\geq j^*$ and apply H\"older's inequality once again with the same exponents
to obtain
\begin{eqnarray}
	&&\sum_{j\geq j^*} \sum_{\xi \in {\nabla}_{j,\mathrm{bnd}}} \abs{\distr{u}{\tilde{\psi}^{\partial\Omega}_{j,\xi}}}^\tau \nonumber\\
	&&\quad \lesssim \left( \sum_{j\geq j^*} 2^{jsp + j(p-2)} \sum_{\xi \in {\nabla}_{j,\mathrm{bnd}}} \abs{\distr{u}{\tilde{\psi}^{\partial\Omega}_{j,\xi}}}^p \right)^{\tau/p} \left( \sum_{j\geq j^*} \left( 2^{1-(sp+p-2)\tau/(p-\tau)} \right)^j \right)^{(p-\tau)/p}.\label{2factors}
\end{eqnarray}
Extending the inner summation from ${\nabla}_{j,\mathrm{bnd}}$ to ${\nabla}_j^{\partial\Omega}$ shows that the first factor in \link{2factors} can be bounded by $\norm{u \sep B^s_p(L_p(\partial\Omega))}^\tau$.
Furthermore, we note that
\begin{equation*}
	1- (sp+p-2) \, \frac{\tau}{p-\tau} < 0 	
	\qquad \text{if and only if} \qquad
	\frac{1}{\tau} < 1 - \frac{1}{p} + s.
\end{equation*}
Thus our assumptions ensure that the second factor in \link{2factors} is bounded by some constant which completes the proof.
\end{proof}

In order to give an efficient treatment of the interior wavelets indexed by $\xi \in {\nabla}_{j,\mathrm{int}}$, i.e., to present a comprehensive proof of \autoref{prop:int}, we need to split this index set further according to the faces $\Gamma^{n,t}$ of the cones $\Co_n$; see \autoref{sect:boundaries}.
For $j\geq j^*$, $n\in\{1,\ldots,N\}$, and $t\in\{1,\ldots, T_n\}$ let us define
\begin{equation}\label{def_nabla_nt}
	{\nabla}_{j,\mathrm{int}}^{n,t}
	:= \left\{ \xi \in {\nabla}_{j,\mathrm{int}} \sep \U_{n,1} \cap \overline{\Gamma^{n,t}} \cap B_{j,\xi} \neq \leer \right\} .
\end{equation}
The main properties of this splitting are described by the subsequent lemma.
\begin{lemma}\label{lem:split}
	Let $j\geq j^*$ and $k\in\{1,2,\ldots,\tilde{d}\}$, $\tilde{d}\in\N$. Then
	\begin{equation}\label{split_int}
		{\nabla}_{j,\mathrm{int}} 
		= \bigcup_{n=1}^N \bigcup_{t=1}^{T_n} {\nabla}_{j,\mathrm{int}}^{n,t}
	\end{equation}
	and for all $\xi \in {\nabla}_{j,\mathrm{int}}^{n,t}$ we have
	\begin{equation}\label{est_Delta}
		2^{-j} < \Delta_{j,\xi}^{n,t} \leq C_3,
	\end{equation}
	where $\Delta_{j,\xi}^{n,t} := \dist{B_{j,\xi}}{\partial\Gamma^{n,t}}$ denotes the distance of the ball $B_{j,\xi}$ to the cone face boundary~$\partial\Gamma^{n,t}$ and $ C_3 := \max \{ \diam \overline{F_i} \sep i=1,\ldots,I \}$.
	Moreover we have the following cancellation property for functions $g\colon\partial\Omega\nach\C$:
	\begin{equation}\label{cancel}
		\abs{\distr{g}{\tilde{\psi}_{j,\xi}^{\partial\Omega}}}
		\lesssim 2^{-kj} \norm{g \sep W^k\!\left(L_2(B_{j,\xi})\right)},
		\qquad j\geq j^*, \quad \xi\in {\nabla}_{j,\mathrm{int}}^{n,t}.
	\end{equation}
	Here derivatives used in the $W^k(L_2(B_{j,\xi}))$-norm have to be understood w.r.t.\ local Cartesian coordinates in the interior of the two-dimensional surface patches $F_i$.
\end{lemma}
\begin{proof}
\emph{Step 1.}
We start by proving the non-trivial inclusion in \link{split_int}.
To this end, let $\xi$ be an element of ${\nabla}_{j,\mathrm{int}} = \bigcup_{i=1}^{I} {\nabla}_{j,\mathrm{int}}^{F_i}$, i.e., $\xi \in {\nabla}_{j,\mathrm{int}}^{F_i}$ for some $i=i^*$. 
Due to \link{B1} and \link{U3} we have
\begin{equation*}
	\leer 
	\neq B_{j,\xi} 
	\subset F_{i^*} 
	\subset \partial\Omega 
	= \bigcup_{n=1}^N \U_{n,1} 
	= \bigcup_{n=1}^N \bigcup_{t=1}^{T_n} \left( \U_{n,1} \cap \overline{\Gamma^{n,t}} \right).
\end{equation*}
Therefore there needs to be at least one pair $(n^*,t^*)$ such that $\U_{n^*,1} \cap \overline{\Gamma^{n^*,t^*}} \cap B_{j,\xi} \neq \leer$. That is, $\xi \in {\nabla}_{j,\mathrm{int}}^{n^*,t^*} \subset \bigcup_{n=1}^N \bigcup_{t=1}^{T_n} {\nabla}_{j,\mathrm{int}}^{n,t}$ as claimed.

\emph{Step 2.}
Next we show \link{est_Delta}. 
Let $\xi \in {\nabla}^{n,t}_{j,\mathrm{int}}$ for some $j$, $n$, $t$ and let $i:=i(n,t)\in\{1,\ldots,I\}$ denote the corresponding patch number to $n$ and $t$; see \link{def_patch}. 
Then the upper bound in \link{est_Delta} directly follows from the fact
\begin{equation*}
	\dist{B_{j,\xi}}{\partial\Gamma^{n,t}} 
	\leq \dist{B_{j,\xi}}{\partial\Gamma^{n,t}\cap \overline{F_{i}}} \leq \diam \overline{F_{i}}.
\end{equation*}

To show the lower bound we note that the first term in
\begin{equation*}
	\dist{B_{j,\xi}}{\partial\Gamma^{n,t}} 
	= \min\!\left\{ \dist{B_{j,\xi}}{\partial\Gamma^{n,t} \cap \partial F_i}, \dist{B_{j,\xi}}{\partial\Gamma^{n,t} \setminus \partial F_i} \right\}
\end{equation*}
is clearly larger than $\dist{B_{j,\xi}}{\partial F_i}$ which is greater than $2^{-j}$ since $\xi\in{\nabla}_{j,\mathrm{int}}^{F_i}$; see \link{B3}.
Thus, it remains to show that the second term
\begin{equation}\label{lowerbound}
	\dist{B_{j,\xi}}{\partial\Gamma^{n,t} \setminus \partial F_i} 
	= \min\!\left\{ \dist{B_{j,\xi}}{\left[\partial\Gamma^{n,t} \setminus \partial F_i\right] \cap F_i}, \dist{B_{j,\xi}}{\partial\Gamma^{n,t} \setminus \overline{F_i}}   \right\}
\end{equation}
of the minimum above, is lower bounded by $2^{-j}$, too.
Observe that for every non-empty set $A$ we have
\begin{equation*}
	\dist{B_{j,\xi}\cap \U_{n,1}\cap \Gamma^{n,t}}{A}
	\leq \dist{B_{j,\xi}}{A} + \diam B_{j,\xi} 
	\leq \dist{B_{j,\xi}}{A} + c\, 2^{-j}
\end{equation*}
since $\diam B_{j,\xi} = 2 \rad B_{j,\xi} \sim \abs{\supp \tilde{\psi}_{j,\xi}^{\partial\Omega}}^{1/2} \sim 2^{-j}$,
due to \link{B2} and \link{supp}. This yields
\begin{equation*}
	\dist{B_{j,\xi}}{A} 
	\geq \dist{F_i\cap \U_{n,1} \cap \Gamma^{n,t}}{A} - c\, 2^{-j}
\end{equation*}
because of $B_{j,\xi} \subset F_i$.
Consequently, for the first term in \link{lowerbound} we conclude
\begin{eqnarray*}
	\dist{B_{j,\xi}}{\left[\partial\Gamma^{n,t} \setminus \partial F_i \right] \cap F_i} 
	&\geq & \dist{\U_{n,1} \cap \Gamma^{n,t}}{\left[\partial\Gamma^{n,t} \setminus \partial F_i\right] \cap F_i} - c\, 2^{-j} \\
	&\geq & \dist{\U_{n,1} \cap \Gamma^{n,t}}{F_i\setminus\Gamma^{n,t}} - c\, 2^{-j} \\
	&\geq & C_1 - c\, 2^{-j} \\
	&>& 2^{-j},
\end{eqnarray*}
where we used $\left[\partial\Gamma^{n,t} \setminus \partial F_i\right] \cap F_i \subset [F_i\setminus\Gamma^{n,t}]$ together with \link{U2} and the fact that $j\geq j^*$.
The second term in \link{lowerbound} can be estimated similarly using \link{U1} instead of \link{U2}:
\begin{eqnarray*}
	\dist{B_{j,\xi}}{\partial\Gamma^{n,t} \setminus \overline{F_i}}
	&\geq & \dist{\U_{n,1}}{\partial\Gamma^{n,t} \setminus \overline{F_i}} - c\, 2^{-j} \\
	&\geq & \dist{\U_{n,1}}{F_\ell} - c\, 2^{-j} \\
	&\geq & C_1 - c\, 2^{-j} \\
	&>& 2^{-j},
\end{eqnarray*}
where $\ell\in\{1,\ldots,I\}$ denotes some patch number with $\nu_n \notin \overline{F_\ell}$.

Hence, \link{lowerbound} is lower bounded by $2^{-j}$ which finally implies the first estimate in \link{est_Delta}.

\emph{Step 3.}
We are left with showing the cancellation property stated in \link{cancel}.
Using the same notation as in the previous step we see that for $\xi\in{\nabla}_{j,\mathrm{int}}^{n,t}$ the support of $\tilde{\psi}_{j,\xi}^{\partial\Omega} = \tilde{\psi}_{j,\xi}^{\square} \circ \kappa_i^{-1}$ is contained in $\kappa_i(Q^\square_{j,\xi}) \subset B_{j,\xi}$ which is a subset of the patch $F_i=F_{i(n,t)}$; see \link{B1}.
Therefore we have
\begin{equation*}
	\distr{g}{\tilde{\psi}_{j,\xi}^{\partial\Omega}} 
	= \distr{g\circ\kappa_i}{\tilde{\psi}_{j,\xi}^{\partial\Omega} \circ \kappa_i}_\square
	= \distr{g\circ\kappa_i - \P}{\tilde{\psi}_{j,\xi}^{\square}}_\square + \distr{\P}{\tilde{\psi}_{j,\xi}^{\partial\Omega} \circ \kappa_i}_\square
\end{equation*}
for all polynomials $\P$ defined on the cube $Q^\square_{j,\xi}$ in the unit square $[0,1]^2$.
If the total degree $\deg \P$ of $\P$ is strictly less than $\tilde{d}$, 
e.g., if $\P\in\Pi_{k-1}(Q^\square_{j,\xi})$ for some $k\in\{1,2,\ldots,\tilde{d}\}$, then the second term equals zero due to the vanishing moment property \link{vanish}; see \autoref{ass:basis}.
The inequality of Cauchy-Schwarz thus yields the estimate
\begin{eqnarray*}
	\abs{\distr{g}{\tilde{\psi}_{j,\xi}^{\partial\Omega}}}
	&\lesssim &\norm{ \left(g\circ\kappa_i-\P\right)\chi_{Q^\square_{j,\xi}} \sep L_2([0,1]^2) } \, \norm{\tilde{\psi}_{j,\xi}^{\square} \sep L_2([0,1]^2)} \\
	&\sim & \norm{ g\circ\kappa_i-\P \sep L_2(Q^\square_{j,\xi}) }
\end{eqnarray*}
since we have  $\norm{\tilde{\psi}_{j,\xi}^{\square} \sep L_2([0,1]^2)} \sim \norm{\tilde{\psi}_{j,\xi}^{\partial\Omega} \sep L_2(\partial\Omega)} \sim 1$ due to the norm equivalence \link{normeq} and the normalization \link{normalized}.
Taking the infimum w.r.t.\ $\P\in\Pi_{k-1}(Q^\square_{j,\xi})$ Whitney's estimate (see \autoref{prop:whitney}) in dimension $2$ with $p=\gamma=2$ gives
\begin{eqnarray*}
	\abs{\distr{g}{\tilde{\psi}_{j,\xi}^{\partial\Omega}}} 
	&\lesssim &\inf_{\P\in\Pi_{k-1}(Q^\square_{j,\xi})} \norm{ g\circ\kappa_i-\P \sep L_2(Q^\square_{j,\xi}) } 
	\lesssim \abs{Q^\square_{j,\xi}}^{k/2} \, \abs{g \circ \kappa_i}_{W^k\left(L_2(Q^\square_{j,\xi})\right)}.
\end{eqnarray*}
From \link{B1}, \link{B2} and \link{supp} it follows that $\abs{Q^\square_{j,\xi}} \sim \abs{\supp{\tilde{\psi}_{j,\xi}^{\partial\Omega}}} \sim 2^{-2j}$. 
Finally, since the parametrizations $\kappa_i$ are assumed to be sufficiently smooth, a simple transformation of measure argument shows the desired estimate for the Sobolev semi-norm:
\begin{equation*}
	\abs{g \circ \kappa_i}_{W^k\left(L_2(Q^\square_{j,\xi})\right)}
	\lesssim \norm{g \sep W^k\!\left(L_2(\kappa_i(Q^\square_{j,\xi}))\right)}
	\lesssim \norm{g \sep W^k\!\left(L_2(B_{j,\xi})\right)}.
\end{equation*}
This completes the proof.
\end{proof}

To show \autoref{prop:int} we split up the index sets ${\nabla}_{j,\mathrm{int}}^{n,t}$ which we defined in~\link{def_nabla_nt} again into disjoint subsets
\begin{equation}\label{split_int_a}
	{\nabla}_{j,\mathrm{int}}^{n,t}(a)
	:= \left\{ \xi \in {\nabla}_{j,\mathrm{int}}^{n,t} \sep a\, 2^{-j} < \Delta_{j,\xi}^{n,t} \leq (a+1)\, 2^{-j} \right\}, \qquad a\in\N.
\end{equation}
Note that due to \link{est_Delta} there are only finitely many $a\in\N$ such that ${\nabla}_{j,\mathrm{int}}^{n,t}(a)$ is not empty.
Furthermore, using standard arguments it is easy to see that
\begin{equation}\label{card_nabla}
	\# {\nabla}_{j,\mathrm{int}}^{n,t}(a) \lesssim 2^j
\end{equation}
with an implied constant that is independent of $j$, $n$, $t$, and $a$.

\begin{proof}[Proof (\autoref{prop:int})]
\emph{Step 1.}
Let $j\geq j^*$ and $\tau>0$ be fixed.
Then \link{resolution} together with \autoref{lem:finite} implies
\begin{equation*}
	\sum_{\xi \in {\nabla}_{j,\mathrm{int}}} \abs{\distr{u}{\tilde{\psi}^{\partial\Omega}_{j,\xi}}}^\tau 
	=\sum_{\xi \in {\nabla}_{j,\mathrm{int}}} \abs{\sum_{n=1}^N \distr{\varphi_n u}{\tilde{\psi}^{\partial\Omega}_{j,\xi}}}^\tau 
	\lesssim \sum_{n=1}^N \sum_{\xi \in {\nabla}_{j,\mathrm{int}}} \abs{\distr{\varphi_n u}{\tilde{\psi}^{\partial\Omega}_{j,\xi}}}^\tau. 
\end{equation*}
Recall that $\supp(\varphi_n u) \subset \U_{n,1}=\bigcup_{t=1}^{T_n}(\U_{n,1} \cap \overline{\Gamma^{n,t}})$ and $\supp \tilde{\psi}^{\partial\Omega}_{j,\xi} \subset B_{j,\xi}$ for $\xi\in {\nabla}_{j,\mathrm{int}}$; see \link{U3}, as well as \link{B1}. Hence, we can use the splittings given in \link{split_int} and \link{split_int_a} to conclude
\begin{equation}\label{sums}
	\sum_{\xi \in {\nabla}_{j,\mathrm{int}}} \abs{\distr{u}{\tilde{\psi}^{\partial\Omega}_{j,\xi}}}^\tau 
	\lesssim \sum_{n=1}^N \sum_{t=1}^{T_n} \sum_{a=1}^\infty \sum_{\xi \in {\nabla}_{j,\mathrm{int}}^{n,t}(a)} \abs{\distr{\varphi_n u}{\tilde{\psi}^{\partial\Omega}_{j,\xi}}}^\tau, \qquad j\geq j^*.
\end{equation}
In the next step we estimate all of these wavelet coefficients individually.

\emph{Step 2.}
According to the sums in \link{sums} let $j$, $n$, $t$, as well as $a$, and $\xi$ be fixed.
Using the cancellation property~\link{cancel} in \autoref{lem:split} for $g:=\varphi_n u$ and the fact that $B_{j,\xi}\subset \Gamma^{n,t}$ we obtain
\begin{equation*}
	\abs{\distr{\varphi_n u}{\tilde{\psi}^{\partial\Omega}_{j,\xi}}}
	\lesssim 2^{-kj} \norm{\varphi_n u \sep W^k\!\left(L_2(B_{j,\xi})\right)}
	\sim 2^{-kj} \norm{(\varphi_n u)_{n,t} \sep W^k\!\left(L_2(R_{n,t}^{-1}(B_{j,\xi}))\right)},
\end{equation*}
where $(\varphi_n u)_{n,t}$ refers to the facewise representation of $\varphi_n u$ defined by the change of coordinates~$R_{n,t}$; see \link{facewise_fkt}.
Next we see that for every $\rho \in [0,k]$ the Sobolev norm on $R_{n,t}^{-1}(B_{j,\xi}) \subset \tilde{\Gamma^{n,t}}$ can be estimated by its weighted analogue:
\begin{eqnarray*}
	\norm{(\varphi_n u)_{n,t} \sep W^k\!\left(L_2(R_{n,t}^{-1}(B_{j,\xi}))\right)}^2
	&=& \sum_{\substack{\alpha\in\N_0^2\\\abs{\alpha}\leq k}} \norm{D_y^\alpha(\varphi_n u)_{n,t} \sep L_2\!\left(R_{n,t}^{-1}(B_{j,\xi})\right) }^2 \\
	&\lesssim & \left( \Delta_{j,\xi}^{n,t} \right)^{2(\rho-k)} \sum_{\abs{\alpha}\leq k} \norm{\delta_{n,t}^{k-\rho}\cdot D_y^\alpha(\varphi_n u)_{n,t} \sep L_2\!\left(R_{n,t}^{-1}(B_{j,\xi})\right) }^2, 
\end{eqnarray*}
where we used that for all $y\in \tilde{\Gamma^{n,t}}$ which belongs to the ball $R_{n,t}^{-1}(B_{j,\xi})$ it is
\begin{eqnarray*}
	\delta_{n,t}(y)  
	&=& \min \! \left\{C_2, \dist{y}{\partial\tilde{\Gamma^{n,t}}} \right\} 
	\geq \min \!\left\{C_2, \dist{R_{n,t}^{-1}(B_{j,\xi})}{\partial\tilde{\Gamma^{n,t}}} \right\}\\
	&=& \min \!\left\{C_2, \dist{B_{j,\xi}}{\partial\Gamma^{n,t}} \right\}
	= \Delta_{j,\xi}^{n,t} \cdot \min \left\{C_2/ \Delta_{j,\xi}^{n,t}, 1 \right\}\\
	&\gtrsim & \Delta_{j,\xi}^{n,t}
\end{eqnarray*}
due to the uniform upper bound $\Delta_{j,\xi}^{n,t}\leq C_3$; see \link{est_Delta} in \autoref{lem:split}.
Thus, we finally conclude that
\begin{equation}\label{single_est}
	\abs{\distr{\varphi_n u}{\tilde{\psi}^{\partial\Omega}_{j,\xi}}}
	\lesssim 2^{-kj} \left( \Delta_{j,\xi}^{n,t} \right)^{\rho-k} \sum_{\abs{\alpha}\leq k} \norm{\delta_{n,t}^{k-\rho}\cdot D_y^\alpha(\varphi_n u)_{n,t} \sep L_2\!\left(R_{n,t}^{-1}(B_{j,\xi})\right) }
\end{equation}
holds true for all $j\geq j^*$, $n\in\{1,\ldots,N\}$, $t\in\{1,\ldots,T_n\}$, $a\in\N$, and $\xi\in{\nabla}_{j,\mathrm{int}}^{n,t}(a)$.

\emph{Step 3.}
Here we combine the splitting \link{sums} from Step 1 with the individual upper bounds~\link{single_est} from the previous step.
By definition of the index sets ${\nabla}_{j,\mathrm{int}}^{n,t}(a)$ -- see \link{split_int_a} -- we know that
$2^{-k j} \left( \Delta_{j,\xi}^{n,t} \right)^{\rho-k} < a^{\rho-k} \cdot 2^{- j  \rho}$.
Hence, for $\tau>0$ and $j\geq j^*$ we obtain
\begin{eqnarray*}
	&&\sum_{\xi \in {\nabla}_{j,\mathrm{int}}} \abs{\distr{u}{\tilde{\psi}^{\partial\Omega}_{j,\xi}}}^\tau \\
	&&\qquad \lesssim 2^{-j  \rho\tau} \sum_{n=1}^N \sum_{t=1}^{T_n} \left( \sum_{a=1}^\infty \sum_{\xi \in {\nabla}_{j,\mathrm{int}}^{n,t}(a)} \! a^{\tau(\rho-k)} \!	\left[ \sum_{\abs{\alpha}\leq k} \norm{\delta_{n,t}^{k-\rho}\cdot D_y^\alpha(\varphi_n u)_{n,t} \sep L_2\!\left(R_{n,t}^{-1}(B_{j,\xi})\right) } \right]^\tau \right).
\end{eqnarray*}
If we restrict ourselves to $1/2 < 1/\tau <\infty$, then the double sum within the brackets can be estimated by H\"older's inequality using the conjugate exponents $2/(2-\tau)$ and $2/\tau$:
\begin{equation}\label{hoelder}
	\big(\ldots\big) 
	\lesssim \left[ \sum_{a,\xi} a^{2\tau(\rho-k)/(2-\tau)} \right]^{(2-\tau)/2} 
	 \left[ \sum_{a,\xi} \sum_{\abs{\alpha}\leq k} \norm{\delta_{n,t}^{k-\rho}\cdot D_y^\alpha(\varphi_n u)_{n,t} \sep L_2\!\left(R_{n,t}^{-1}(B_{j,\xi})\right) }^2 \right]^{\tau/2}.
\end{equation}
The triple sum in \link{hoelder} can be bounded by
\begin{equation*}
	\sum_{\abs{\alpha}\leq k} \sum_{a=1}^\infty \sum_{\xi \in {\nabla}_{j,\mathrm{int}}^{n,t}(a)} \norm{\delta_{n,t}^{k-\rho}\cdot D_y^\alpha(\varphi_n u)_{n,t} \sep L_2\!\left(R_{n,t}^{-1}(B_{j,\xi})\right) }^2
	\lesssim \sum_{\abs{\alpha}\leq k} \norm{\delta_{n,t}^{k-\rho}\cdot D_y^\alpha(\varphi_n u)_{n,t} \sep L_2\!\left(\tilde{\Gamma^{n,t}}\right) }^2
\end{equation*}
because ${\nabla}_{j,\mathrm{int}}^{n,t} = \bigcup_{a=1}^\infty {\nabla}_{j,\mathrm{int}}^{n,t}(a)$ and every point $y\in\tilde{\Gamma^{n,t}}$ only appears in a uniformly bounded number of balls $R_{n,t}^{-1}(B_{j,\xi})$, $\xi\in {\nabla}_{j,\mathrm{int}}^{n,t}$. 
The reason is that, by the properties of the dual wavelets, the same is true for every point $x\in \Gamma^{n,t} \cap F_{i(n,t)}\subset \partial\Omega$ and the sets $\supp \tilde{\psi}^{\partial\Omega}_{j,\xi}$, $\xi\in {\nabla}_{j,\mathrm{int}}^{n,t}$; see \link{finite_overlap} in  \autoref{ass:basis}. Now \autoref{lem:weighted_Sobolev} applied to the function $f_n:=\varphi_n u$ (which is, by definition, supported in $\overline{\U_{n,1}} \subset \partial\Co_n$) allows us to extend the last inequality to
\begin{equation*}
	\sum_{a,\xi} \sum_{\abs{\alpha}\leq k} \norm{\delta_{n,t}^{k-\rho}\cdot D_y^\alpha(\varphi_n u)_{n,t} \sep L_2\!\left(R_{n,t}^{-1}(B_{j,\xi})\right) }^2
	\lesssim \norm{ \varphi_n u \sep X_\rho^k(\partial\Co_n)}^2.
\end{equation*}
	
So let us turn to the first double sum in \link{hoelder}.
From \link{card_nabla} it follows that
\begin{equation*}
	\sum_{a=1}^\infty \sum_{\xi \in {\nabla}_{j,\mathrm{int}}^{n,t}(a)} a^{2\tau(\rho-k)/(2-\tau)}
	\lesssim 2^j \sum_{a=1}^\infty a^{2\tau(\rho-k)/(2-\tau)}
	\lesssim 2^j
\end{equation*}
since the exponent of $a$ is strictly smaller than $-1$ if and only if $1/\tau < 1/2 + k - \rho$ which in turn directly follows from our assumptions.

Combining all the estimates from this step leads to
\begin{eqnarray*}
	\sum_{\xi \in {\nabla}_{j,\mathrm{int}}} \abs{\distr{u}{\tilde{\psi}^{\partial\Omega}_{j,\xi}}}^\tau
	&\lesssim & 2^{-j \rho\tau} \sum_{n=1}^N \sum_{t=1}^{T_n} 2^{j(2-\tau)/2} \norm{ \varphi_n u \sep X_\rho^k(\partial\Co_n)}^\tau \\
	&\lesssim & \left( 2^{1-\tau(\rho+1/2)} \right)^j \sum_{n=1}^N \norm{ \varphi_n u \sep X_\rho^k(\partial\Co_n)}^\tau \\
	&\sim & \left( 2^{1-\tau(\rho+1/2)} \right)^j \norm{ u \sep X_\rho^k(\partial\Omega)}^\tau
	\qquad \text{for all} \qquad j\geq j^*,
\end{eqnarray*}
where we used \link{def_X} from the definition of $X_\rho^k(\partial\Omega)$, as well as \autoref{lem:finite}, to obtain the last line.
Note that the latter estimate remains valid also for $\tau=2$. In this case we do not need H\"{o}lder's inequality since we may estimate $a^{\tau(\rho-k)}$ simply by one.

In order to prove the claim~\link{est_int}, we now only need to check that 
\begin{equation}\label{final_cond}
	2^{1-\tau(\rho+1/2)} < 1,
\end{equation}
since then the sum over $j\geq j^*$ will converge.
But the condition~\link{final_cond} is obviously fulfilled, because it is equivalent to our assumption $1/\tau< 1/2 + \rho$.
\end{proof}
\end{appendix}

\section*{Acknowledgements}
\addcontentsline{toc}{section}{Acknowledgments}
We like to thank H.~Harbrecht, B.~Scharf, as well as C.~Hartmann, and F.~Eckhardt for valuable discussions and remarks.
Moreover, we are grateful to the two anonymous referees for several comments which helped to improve the paper.

\addcontentsline{toc}{chapter}{References}
\bibliographystyle{is-abbrv}

\section*{}
\noindent Stephan~Dahlke and Markus~Weimar \\
Philipps-University Marburg \\
Faculty of Mathematics and Computer Science, Workgroup Numerics and Optimization \\
Hans-Meerwein-Stra{\ss}e, Lahnberge \\
35032 Marburg, Germany \\
\{dahlke, weimar\}@mathematik.uni-marburg.de \\[1ex]
\end{document}